\documentclass[11pt, twoside, leqno]{article}
\usepackage{amssymb,amsmath,amsfonts,amsthm,mathrsfs}
\usepackage{indentfirst}
\usepackage{graphics}
\usepackage{color}
\usepackage[margin=3cm]{geometry}
\usepackage{enumerate}
\usepackage{amsmath}
\usepackage{bbm}           
\usepackage{bm}
\usepackage{eso-pic,graphicx}
\usepackage{tikz}
\usepackage{cite}
\usepackage{esint}
\usepackage[colorlinks=true, pdfstartview=FitV, linkcolor=cyan, citecolor=magenta, urlcolor=magenta]{hyperref}
\usepackage{pslatex}
\usepackage{rotating}
\usepackage{verbatim}

\makeatletter
\def\Ddots{\mathinner{\mkern1mu\raise\p@
		\vbox{\kern7\p@\hbox{.}}\mkern2mu
		\raise4\p@\hbox{.}\mkern2mu\raise7\p@\hbox{.}\mkern1mu}}
\makeatother

\def\XXint#1#2#3{{\setbox0=\hbox{$#1{#2#3}{\int}$}
		\vcenter{\hbox{$#2#3$}}\kern-.5\wd0}}

\allowdisplaybreaks

\begin{document}
	\def\rn{{\mathbb R^n}}  \def\sn{{\mathbb S^{n-1}}}
	\def\co{{\mathcal C_\Omega}}
	\def\z{{\mathbb Z}}
	\def\nm{{\mathbb (\rn)^m}}
	\def\mm{{\mathbb (\rn)^{m+1}}}
	\def\n{{\mathbb N}}
	\def\cc{{\mathbb C}}
	
	\newtheorem{defn}{Definition}[section]
	\newtheorem{thm}{Theorem}[section]
	\newtheorem{lem}{Lemma}[section]
	\newtheorem{cor}{Corollary}[section]
	\newtheorem{rem}{Remark}[section]
	\newtheorem{pro}{Proposition}[section]
	\newtheorem{que}{Question}
	
	\renewcommand{\theequation}{\arabic{section}.\arabic{equation}}

	\title
	{\bf\Large Characterizations for multi-sublinear operators and their commutators on three kinds of generalized weighted Morrey spaces and applications
\footnotetext{
			\newline\indent\hspace{1mm} {\it 2020 Mathematics Subject Classification}: Primary 42B20; Secondary 42B25, 42B35, 47H60, 47B47.}}
	
	\date{}
	\author{Xi Cen\footnote{Corresponding author, E-mail: xicenmath@gmail.com}, Xiang Li and Dunyan Yan}
	\maketitle
	
\begin{center}
\begin{minipage}{13cm}
{\small {\bf Abstract:}\quad
The main questions raised in this paper are to find the sufficient conditions that make multi-sublinear operators $T$ and their commutators ${T_{\prod \vec b }}$, ${T_{\sum {\vec b} }}$ to be bounded on three kinds of generalized weighted Morrey spaces. We give the main theorems of this paper to solve the above related questions. As corollaries of the main theorems, we give sufficient and necessary conditions for a class of multi-sublinear operators which are bounded on three kinds of generalized weighted Morrey spaces. As some inportant applications, we apply the main results to the multilinear vector-valued Calder\'on-Zygmund operators, multilinear Littlewood-Paley square operators, multilinear pseudo-differential operators and multilinear paraproducts.}
\end{minipage}
\end{center}
\section{Introduction}\label{sec1}
The theory of multilinear Calder\'on-Zygmund operators and multilinear Littlewood-Paley operators have played very important roles in modern harmonic analysis with lots of extensive applications in the others fields of mathematics, which have been extensively studied. The multilinear Calderón–Zygmund theory was first studied by Coifman and Meyer, whose study was motivated not only as generalizations of the theory of linear ones but also its natural appearance in harmonic analysis. Recently, this topic has received increasing attentions and well development and a lot of research work involves these operators from various points of view. see \cite{Lerner,xue1,xue2,xue3,xue4,xue5,LuZhang,Zhang,Janson,Ismayilova,Lin-Yan,LuYangZhou,Xu,Huang-Xu,Hu-Li-Wang} for details.

We now firstly recall the definitions of multilinear Calderón-Zygmund operators with Dini kernel.
\begin{defn}[\cite{LuZhang}]
	Suppose that $\theta :[0, + \infty ) \to [0, + \infty )$ is a nondecreasing function with $0 < \theta(1) < \infty.$ For $a>0,$ we say $\theta  \in Dini(a),$ if $${\left[ \theta  \right]_{Dini(a)}} = \int_0^1 {\frac{{{\theta ^a}(t)}}{t}} dt < \infty. $$
\end{defn}
\begin{defn}[\cite{LuZhang}]
	For any $t \in (0,\infty ),$ let ${K}(x,{y_1}, \cdots ,{y_m})$ be a locally integrable function defined away from the diagonal $x = {y_1} =  \cdots  = {y_m}$ in $(\mathbb R^n)^{m+1}$. We say $K$ is a kernel of type $\theta$ if for some constants $A>0,$ such that 
	\begin{enumerate}
		\item[\emph{(1)}]$
		{\left| {K(x,\vec y)} \right|} \le \frac{A}{{{{(\sum\limits_{j = 1}^m {\left| {x - {y_j}} \right|} )}^{mn}}}};
		$\item[\emph{(2)}]$
		{\left| {{K}(x,\vec y) - {K}(x,{y_1}, \cdots ,{y_i}^\prime , \cdots ,{y_m})} \right|} \le \frac{A}{{{{(\sum\limits_{j = 1}^m {\left| {x - {y_j}} \right|} )}^{mn}}}} \cdot \theta (\frac{{{{\left| {{y_i} - {{y'}_i}} \right|} }}}{{\sum\limits_{j = 1}^m {\left| {x - {y_j}} \right|} }}),$;
		\item[\emph{(3)}]$
		{\left| {{K}(z,\vec y) - {K}(x,\vec y)} \right|} \le \frac{A}{{{{(\sum\limits_{j = 1}^m {\left| {x - {y_j}} \right|} )}^{mn}}}} \cdot \theta (\frac{{\left| {z - x} \right|}}{{\sum\limits_{j = 1}^m {\left| {x - {y_j}} \right|} }})$,
	\end{enumerate}
	where $(2)$ holds for any $i \in \{ 1, \cdots ,m\}$, whenever $\left| {{y_i} - {y_i}^\prime } \right| \le  \frac{1}{2}\mathop {\max }\limits_{1 \le j \le m} \{ \left| {x - {y_j}} \right|\}$
	and $(3)$ holds whenever $\left| {z - x} \right| \le \frac{1}{2} \mathop {\max }\limits_{1 \le j \le m} \{ \left| {x - {y_j}} \right|\}$.
\end{defn}
When $\theta \left( t \right) = {t^\gamma }$ for some $\gamma  > 0$, we say $K$ is a $m$-linear Calder\'on-Zygmund kernel.

We say $T: {\mathscr S}({\rn}) \times  \cdots  \times {\mathscr S}({\rn}) \to {\mathscr S}'({\rn})$ is an $m$-linear Calder\'on-Zygmund operator with kernel $K$ if
\begin{equation*}\label{NW8}
T(\vec f)(x) = \int_{{\nm}} {K(x,\vec y)\prod\limits_{j = 1}^m {{f_j}({y_j})d{\vec y}}},
\end{equation*}
for any $\vec f \in {C_c^\infty}({\rn}) \times  \cdots  \times {C_c^\infty}({\rn})$ and any $x \notin \bigcap\limits_{j = 1}^m {{\rm{supp}}{f_j}}$, and $T$ can be extended to be a bounded operator from ${L^{{q_1}}} \times  \cdots \times{L^{{q_m}}}$ to ${L^q}$, for some $1 \le {q_1} \cdots , {q_m} < \infty, \frac{1}{q} = \sum\limits_{k = 1}^m {\frac{1}{{{q_k}}}}$.

$T$ is called a $m$-linear Calder\'on-Zygmund operator with Dini kernel $K$ when $K$ is a kernel of type $\theta \in Dini(1)$.

Let $T$ be a $m$-sublinear operator, for $\vec{b}=(b_1,\cdots,b_m) \in {(L_{loc}^1)^m}$, the $m$-sublinear commutator of $T$ and $\vec{b}$ is defined by
\begin{equation*}
{T_{\Sigma \vec b}}({f_1}, \cdots ,{f_m})(x) =\sum\limits_{j = 1}^m {T_{\vec b}^j} (\vec f)(x): = \sum\limits_{j = 1}^m {T({f_1}, \cdots ,({b_j}(x) - {b_j}){f_j}, \cdots {f_m})(x)};
\end{equation*}
the iterated commutator of $T$ and $\vec{b}$ is defined by
\begin{equation*}
{T_{\Pi \vec b}}(\vec f)(x) = T(({b_1}(x) - {b_1}){f_1}, \cdots ,({b_m}(x) - {b_m}){f_m})(x).
\end{equation*}
Now, we give some definitions of weights and some important weighted spaces.
\begin{defn}
Let $\omega$ be a weight function on a measurable set $E \subseteq \rn$ and $1 \le p<\infty$. The weighted Lebesgue spaces are defined by
	\begin{equation*}
{L^p}(E,\omega dx) = \{ f:{\left\| f \right\|_{{L^p}(E,\omega dx)}}: = {\left( {\int_E {{{\left| {f(x)} \right|}^p}\omega (x)dx} } \right)^{\frac{1}{p}}} < \infty \}.
	\end{equation*}
	The weak weighted Lebesgue spaces are defined by
	\begin{equation*}
W{L^p}(E,\omega dx) = \{ f:{\left\| f \right\|_{W{L^p}(E,\omega dx)}}: = \mathop {\sup }\limits_{\alpha  > 0} \alpha \cdot \omega {\left( {\{ x \in E:\left| {f(x)} \right| > \alpha \} } \right)^{\frac{1}{p}}} < \infty \}.
	\end{equation*}
When $p= \infty$, 	
\begin{equation*}
{L^\infty }(E,\omega dx) = W{L^\infty }(E,\omega dx) = \mathop {ess\sup }\limits_{x \in E} \left| {f(x)} \right|\omega (x).
\end{equation*}
For simplicity, we abbreviate ${L^p}(\rn,\omega dx)$ to ${L^p}(\omega )$ and $W{L^p}(\rn,\omega dx)$ to $W{L^p}(\omega )$.
\end{defn}
The classical Morrey spaces $L^{p,\lambda}$ were first introduced by Morrey in \cite{Morrey} to study the local behavior of solutions to second order elliptic partial differential equations. In 1998, Lu, Yang and Zhou \cite{LuYangZhou} studied the sublinear operators with rough kernel on generalized Morrey spaces. In 2009, Komori and Shirai \cite{Komori} considered the weighted version of Morrey spaces $L^{p,\kappa}(\omega)$ and studied the boundedness of some classical operators such as the Hardy-Littlewood maximal operator and the Calder\'on-Zygmund operator on these spaces. In the same year, Guliyev \cite{Gu0} first studied the boundedness of the maximal, potential and singular operators on the generalized Morrey spaces. This greatly promotes the mathematical workers to study the Morrey spaces and singular integral operators.

In 2012, Guliyev \cite{Gu1} proved boundedness of higher order commutators of sublinear operators on generalized weighted Morrey spaces. 
In 2014, Hu, Li and Wang \cite{Hu-Wang, Hu-Li-Wang} studied the multilinear singular integral operators and multilinear fractional integral operators on generalized weighted Morrey spaces.
In 2021, Ismayilova \cite{Ismayilova} studied Calderón-Zygmund operators with kernels of Dini’s type and their multilinear commutators on generalized Morrey spaces. In the same year, Lin and Yan \cite{Lin-Yan} considered the multilinear strongly singular Calderón-Zygmund operators and commutators on Morrey type spaces.
In 2022, Guliyev \cite{Gu3} proved the boundedness of multilinear Calderón–Zygmund operators with kernels of Dini’s type and their commutators on generalized local Morrey spaces. 
In 2023, Guliyev \cite{Gu2} obtained the boundedness of commutators of multilinear Calderón–Zygmund operators with kernels of Dini’s type on generalized weighted Morrey spaces and applications and Cen \cite{cen1} proved boundedness of multilinear Littlewood-Paley square operators and their commutators on weighted Morrey spaces.

Let us recall the following definitions of generalized weighted Morrey spaces and generalized local weigthed Morrey spaces.

\begin{defn}[\cite{Gu1}]
	Let $1 \le p<\infty$, $\omega$ be a weight function on $\mathbb R^n$ and $\varphi$ be a positive measurable function on $\rn \times \left( {0,\infty } \right)$. The generalized weighted Morrey spaces are defined by
	\begin{equation*}
{M^{p,\varphi }}(\omega ) = \{ f:{\left\| f \right\|_{{M^{p,\varphi }}(\omega )}}: = \mathop {\sup }\limits_{x \in \rn,r > 0} \varphi {(x,r)^{ - 1}}\omega {(B(x,r))^{ - \frac{1}{p}}}{\left\| f \right\|_{{L^p}(B(x,r),\omega dx)}} < \infty\}.
	\end{equation*}
	The weak generalized weighted Morrey spaces are defined by
	\begin{equation*}
W{M^{p,\varphi }}(\omega ) = \{ f:{\left\| f \right\|_{{M^{p,\varphi }}(\omega )}}: = \mathop {\sup }\limits_{x \in \rn,r > 0} \varphi {(x,r)^{ - 1}}\omega {(B(x,r))^{ - \frac{1}{p}}}{\left\| f \right\|_{W{L^p}(B(x,r),\omega dx)}}< \infty\}.
	\end{equation*}
\end{defn}

\begin{rem}
	
	$(1)~$ If $\omega \equiv 1$, then $M^{p,\varphi}(1)=M^{p,\varphi}$ are the generalized Morrey spaces and $WM^{p,\varphi}(1)=WM^{p,\varphi}$ are the weak generalized Morrey spaces.
	
	$(2)~$ If $\varphi(x,r) \equiv \omega(B(x,r))^{\frac{\kappa-1}{p}}$, then
	$M^{p,\varphi}(\omega)=L^{p,\kappa}(\omega)$ is the weighted Morrey spaces.
	
	$(3)~$ If $\varphi(x,r) \equiv v(B(x,r))^{\frac{\kappa}{p}} \omega(B(x,r))^{-\frac{1}{p}}$, then
	$M^{p,\varphi}_{}(\omega)=L^{p,\kappa}_{}(v,\omega)$ is the two weighted Morrey spaces.
	
	$(4)~$ If $w\equiv1$ and $\varphi(x,r)=r^{\frac{\lambda-n}{p}}$ with $0<\lambda<n$, then $M^{p,\varphi}(\omega)=L^{p,\lambda}$ is the Morrey spaces and $WM^{p,\varphi}(\omega)=WL^{p,\lambda}$ is the weak Morrey spaces.
	
	$(5)~$ If $\varphi(x,r) \equiv \omega(B(x,r))^{-\frac{1}{p}}$, then $M^{p,\varphi}(\omega)=L^{p}(\omega)$ is the weighted Lebesgue spaces.
\end{rem}

\begin{defn}
	Let $1 \le p<\infty$, $\omega$ be a weight function on $\mathbb R^n$ and $\varphi$ be a positive measurable function on $\rn \times \left( {0,\infty } \right)$. The generalized local weighted Morrey spaces are defined by
	\begin{equation*}
	M_{{x_0}}^{p,\varphi }(\omega ) = \{ f:{\left\| f \right\|_{M_{{x_0}}^{p,\varphi }(\omega )}}: = \mathop {\sup }\limits_{r > 0} \varphi {({x_0},r)^{ - 1}}\omega {(B({x_0},r))^{ - \frac{1}{p}}}{\left\| f \right\|_{{L^p}(B({x_0},r),\omega dx)}} < \infty \}.
	\end{equation*}
	The weak generalized local weighted Morrey spaces are defined by
	\begin{equation*}
	WM_{{x_0}}^{p,\varphi }(\omega ) = \{ f:{\left\| f \right\|_{WM_{{x_0}}^{p,\varphi }(\omega )}}: = \mathop {\sup }\limits_{r > 0} \varphi {({x_0},r)^{ - 1}}\omega {(B({x_0},r))^{ - \frac{1}{p}}}{\left\| f \right\|_{{WL^p}(B({x_0},r),\omega dx)}} < \infty \}.
	\end{equation*}
\end{defn}
With the development of variable exponent Lebesgue spaces, the generalized variable exponent Morrey spaces have been studied by some mathematical workers in recent years. From 2018-2021, Guliyev \cite{Gu5,Gu6,Gu7} studied the Maximal and singular integral operators and their commutators, Calderón-Zygmund operators with kernels of Dini’s type on generalized weighted Morrey spaces with variable exponent and also proved the boundedness of Calderón-Zygmund operators with kernels of Dini’s type and their multilinear commutators on generalized variable exponent morrey spaces. In 2022, Xu \cite{Xu} gave the boundedness of bilinear $\theta $-type Calderón-Zygmund operators and its commutators on generalized variable exponent Morrey spaces.

In order to present the main results of this paper, we start giving some conceptions for weigthed variable exponent Lebesgue spaces and generalized weigthed variable exponent Morrey spaces.

Given an open set $E \subseteq {\rn}$ and a measurable function $p( \cdot ): {\rm{E}} \to \left[ {1,\infty } \right)$, $p'( \cdot )$ is the conjugate exponent defined by $p'( \cdot ) = \frac{{p( \cdot )}}{{p( \cdot ) - 1}}$. For a measurable subset $E \subseteq {\rn}$, we denote ${p^ - }(E)= \mathop {ess\inf }\limits_{x \in E} \{ p(x)\}, {p^ + }(E) =\mathop {ess\sup }\limits_{x \in E} \{ p(x)\}$. Especially, we denote ${p^ - }={p^ - }(\rn)$ and ${p^ + }={p^ + }(\rn)$. We give several sets of measurable functions as follows.
\begin{align*}
&{\mathcal P}\left( E \right) = \{ p( \cdot ) :{\rm{E}} \to \left[ {1,\infty } \right) \text{ is measurable: } 1 < {p^ - }(E) \le {p^{\rm{ + }}}(E) < \infty \};\\
&{{\mathcal P}_1}\left( {E} \right) = \{ p( \cdot ) :{\rm{E}} \to \left[ {1,\infty }\right) \text{ is measurable: } 1 \le {p^ - }(E) \le {p^{\rm{ + }}}(E) < \infty \};\\
&{{\mathcal P}_0}\left( {E} \right) = \{ p( \cdot ) :{\rm{E}} \to \left[ {0,\infty }\right) \text{ is measurable: } 0 < {p^ - }(E) \le {p^{\rm{ + }}}(E) < \infty \}.
\end{align*}
Obviously, we have ${\mathcal P}\left( E \right) \subseteq {{\mathcal P}_1}\left( {E} \right) \subseteq {{\mathcal P}_0}\left( {E} \right)$. When $E=\rn$, we take the shorthand. For example ${\mathcal P}\left( \rn \right)$, we write it by ${\mathcal P}$.

\begin{defn}[\cite{Gu5}]
	Let $p( \cdot )$ is a measurable function. We define the variable exponent Lebesgue spaces with Luxemburg norm by
	\begin{equation*}
	L_{}^{p( \cdot )}(E) = \{ f:{\left\| f \right\|_{L_{}^{p( \cdot )}(E)}} := \inf \{ \lambda  > 0:{\int_E^{} {\left( {\frac{{\left| {f(x)} \right|}}{\lambda }} \right)} ^{p(x)}}dx \le 1\}  < \infty \}.
	\end{equation*}
	For a open set $\Omega \subseteq \rn$, we define the locally variable exponent Lebesgue spaces by
	\begin{equation*}
	L_{loc}^{p( \cdot )}(\Omega) = \{ f:f \in L_{}^{p( \cdot )}(E) \text{ for all compact subsets } E \subseteq \Omega \}.
	\end{equation*}
	Let $\omega$ be a weight function on $E$. The variable exponent weighted Lebesgue spaces are defined by
	\begin{equation*}
	L_{}^{p( \cdot )}(E,\omega dx) = \{ f:{\left\| f \right\|_{L_{}^{p( \cdot )}(E,\omega dx)}} = {\left\| {\omega f} \right\|_{L_{}^{p( \cdot )}(E)}} < \infty \}.
	\end{equation*}
\end{defn}
We define a important set $\mathcal{B}$ by
\begin{equation*}
\mathcal{B} := \{ p( \cdot ) \in {\mathcal P}:\text{ the Hardy-Littlewood maximal operator }M \in B(L_{}^{p( \cdot )} \to L_{}^{p( \cdot )}))\}.
\end{equation*}
We say $p( \cdot ) \in LH(\rn)$(globally log-H\"{o}lder continuous functions), if $p( \cdot )$ satisfies
\begin{align*}
&\left| {p(x) - p(y)} \right| \le \frac{C}{{ - \log \left( {\left| {x - y} \right|} \right)}}, when \left| {x - y} \right| \le \frac{1}{2}, and\\
&\left| {p(x) - p(y)} \right| \le \frac{C}{{\log \left( {e + \left| x \right|} \right)}}, when \left| y \right| \ge \left| x \right|.
\end{align*}
The variable exponent ${A_{p( \cdot )}}\left( {\rn} \right)$ are defined by
\begin{equation*}
{A_{p( \cdot )}}\left( {\rn} \right) = \{ \omega \text{ is a weight: }{\left[ \omega  \right]_{{A_{p( \cdot )}}}} = \mathop {\sup }\limits_B {\left| B \right|^{ - 1}}{\left\| \omega  \right\|_{L_{}^{p( \cdot )}(B(x,r))}}{\left\| {{\omega ^{ - 1}}} \right\|_{L_{}^{p'( \cdot )}(B(x,r))}} < \infty \} 
\end{equation*}

\begin{defn}[\cite{Gu5}]
	Let $p( \cdot ) \in {\mathcal P}_1$, $\omega$ be a weight function on $\mathbb R^n$ and $\varphi$ be a positive measurable function on $\rn \times \left( {0,\infty } \right)$. The generalized weighted variable exponent Morrey spaces are defined by
	\begin{equation*}
	{M^{p( \cdot ),\varphi }}(\omega ) = \{ f:{\left\| f \right\|_{{M^{p( \cdot ),\varphi }}(\omega )}}: = \mathop {\sup }\limits_{x \in {\rn},r > 0} \varphi {(x,r)^{ - 1}}\left\| \omega  \right\|_{L_{}^{p( \cdot )}(B(x,r))}^{ - 1}{\left\| f \right\|_{{L^{p( \cdot )}}(B(x,r),\omega dx)}} < \infty \}.
	\end{equation*}
	where ${\left\| f \right\|_{{L^{p( \cdot )}}(B(x,r),\omega dx)}} \equiv {\left\| {f{\chi _{B(x,r)}}} \right\|_{{L^{p( \cdot )}}(\omega )}}$.
\end{defn}

\begin{rem}
	
	$(1)~$ If $\omega \equiv 1$, then ${M^{p( \cdot ),\varphi }}(\omega )(1)={M^{p( \cdot ),\varphi }}(\omega )$ are the generalized variable exponent Morrey spaces.
	
	$(2)~$ If $\varphi(x,r) \equiv \omega(B(x,r))^{\frac{\kappa-1}{p(x)}}$, then
	${M^{p( \cdot ),\varphi }}(\omega )={L^{p( \cdot ),\kappa }}(\omega )$ is the weighted variable exponent Morrey spaces.
	
	$(3)~$ If $\varphi(x,r) \equiv v(B(x,r))^{\frac{\kappa}{p(x)}} \left\| \omega  \right\|_{L_{}^{p( \cdot )}(B(x,r))}^{ - 1}$, then
	$M^{p(\cdot),\varphi}_{}(\omega)=L^{p(\cdot),\kappa}_{}(v,\omega)$ is the two weighted variable exponent Morrey spaces.
	
	$(4)~$ If $w\equiv1$ and $\varphi(x,r)=r^{\frac{\lambda-n}{p(x)}}$ with $0<\lambda<n$, then $M^{p(\cdot),\varphi}(\omega)=L^{p(\cdot),\lambda}$ is the variable exponent Morrey spaces.
	
	$(5)~$ If $\varphi(x,r) \equiv \left\| \omega  \right\|_{L_{}^{p( \cdot )}(B(x,r))}^{ - 1}$, then $M^{p(\cdot),\varphi}(\omega)=L^{p(\cdot)}(\omega)$ is the weighted variable exponent Lebesgue spaces.
\end{rem}

In this paper, we mainly address the following two questions.
\begin{que}\label{que1}
What conditions guarantee boundedness of the multi-sublinear operators and their commutators on generalized weighted Morrey spaces, generalized local weighted Morrey spaces and generalized weighted variable exponent Morrey spaces?
\end{que}

\begin{que}\label{que2}
If the above conditions exist, are there any operator in harmonic analysis that satisfies the above conditions and thus has the corresponding operator boundedness?
\end{que}
To solve these questions, we give the following crucial definitions.
\begin{defn}
Let (quasi or semi) normed spaces $X_i \subseteq L_{loc}^1$, $i = 1, \cdots ,m$, $T$ is an m-sublinear operator from $\prod\limits_{i = 1}^m X_i$ to $L_{loc}^1$. We say m-sublinear operator $T \in LS\left({\prod\limits_{i = 1}^m X_i } \right)$, if T satisfies the local size condition:
for any ball $B \subseteq \rn$, $f_i \in X_i $, $1 \le l \le m$, the following inequality holds 
\begin{equation}\label{cen1}
{\left\| {T(f_1^\infty , \ldots ,f_\ell ^\infty ,f_{\ell  + 1}^0, \ldots ,f_m^0)} \right\|_{{L^\infty }\left( B \right)}} \lesssim \sum\limits_{j = 1}^\infty  {\prod\limits_{i = 1}^m {\frac{1}{{\left| {{2^{j + 1}}B} \right|}}} \int_{{2^{j + 1}}B} {\left| {{f_i}({y_i})} \right|d{y_i}} },
\end{equation}
where $f_i^0 = {f_i}{\chi _{2B}},f_i^\infty  = {f_i}{\chi _{{{\left( {2B} \right)}^c}}}$.
\end{defn}

\begin{defn}
Let (quasi or semi) normed spaces ${X_i}, Y \subseteq L_{loc}^1$, $i = 1, \cdots ,m$, $T$ is an m-sublinear operator from $\prod\limits_{i = 1}^m X_i$ to $L_{loc}^1$. We say m-sublinear operator $T \in LB\left( {\prod\limits_{i = 1}^m {{X_i}}  \to Y} \right)$ if T satisfies the local boundedness condition:
	for any ball $B \subseteq \rn$, $f_i \in {X_i}$, the following inequality holds 
	\begin{equation}\label{cen2}
{\left\| T(f_1^0, \ldots ,f_m^0) \right\|_Y} \lesssim \prod\limits_{i = 1}^m {{{\left\| {{f_i}} \right\|}_{{X_i}}}}.
	\end{equation}
where $f_i^0 = {f_i}{\chi _{2B}}$.
\end{defn}
Now, we introduce the main results of this paper. we shall present the main results on three kinds of generalized weighted Morrey spaces and answer the previous questions.

\begin{center}
\textbf{1: Generalized weighted Morrey spaces}
\end{center}
\begin{thm}\label{thm1}
Let $m\in \mathbb{N}$, $1\leq p_k<\infty$, $k=1,2,\ldots,m$ with $1/p=\sum_{k=1}^m 1/{p_k}$, $\vec{\omega}=(\omega_1,\ldots,\omega_m) \in {A_{\vec P}} \cap {\left( {{A_\infty }}\right)^m}$, $v$ is a weight and a group of non-negative measurable functions $({{\vec \varphi }_1},{\varphi _2}) = ({\varphi _{11}}, \ldots ,{\varphi _{1m}},{\varphi _2})$ satisfy the condition:
\begin{equation}\label{con3}
{\left[ {{{\vec \varphi }_1},{\varphi _2}} \right]_1}: = \mathop {\sup }\limits_{x \in \rn,r > 0} {\varphi _2}{(x,r)^{{\rm{ - }}1}}\int_r^\infty  {\frac{{\mathop {{\rm{essinf}}}\limits_{t < \eta  < \infty } \prod\limits_{i = 1}^m {{\varphi _{1i}}(x,\eta ){\omega _i}{{(B(x,\eta ))}^{\frac{1}{{{p_i}}}}}} }}{{\prod\limits_{i = 1}^m {\omega {{(B(x,t))}^{\frac{1}{{{p_i}}}}}} }}\frac{{dt}}{t}}  < \infty.
\end{equation}
Set $T$ is an m-sublinear operator on $\prod\limits_{i = 1}^m {M^{{p_i},{\varphi _{1i}}}}\left( {{\omega _i}} \right)$, which satisfies 
\begin{equation}\label{cen3}
T \in LS\left( {\prod\limits_{i = 1}^m {{M^{{p_i},{\varphi _{1i}}}}\left( {{\omega _i}} \right)} } \right)\cap LB\left( {\prod\limits_{i = 1}^m {{M^{{p_i},{\varphi _{1i}}}}\left( {{\omega _i}} \right)}  \to {M^{p,{\varphi _2}}}\left( {{v}} \right)} \right).
\end{equation}
\begin{enumerate}[(i)]
\item If $\mathop {\min }\limits_{1 \le k \le m} \{ {p_k}\}  > 1$, then T is bounded from ${M^{{p_1},{\varphi _{11}}}}({\omega _1}) \times  \cdots  \times {M^{{p_m},{\varphi _{1m}}}}({\omega _m})$ to ${M^{p,{\varphi _2}}}({v})$, i.e., $T \in B\left( {\prod\limits_{i = 1}^m {{M^{{p_i},{\varphi _{1i}}}}\left( {{\omega _i}} \right)}  \to {M^{p,{\varphi _2}}}\left( {{v}} \right)} \right);$
\item If $\mathop {\min }\limits_{1 \le k \le m} \{ {p_k}\}  = 1$, then T is bounded from ${M^{{p_1},{\varphi _{11}}}}({\omega _1}) \times  \cdots  \times {M^{{p_m},{\varphi _{1m}}}}({\omega _m})$ to ${WM^{p,{\varphi _2}}}({v})$, i.e., $T \in B\left( {\prod\limits_{i = 1}^m {{M^{{p_i},{\varphi _{1i}}}}\left( {{\omega _i}} \right)}  \to {WM^{p,{\varphi _2}}}\left( {{v}} \right)} \right)$.
\end{enumerate}
\end{thm}

Next, we give the following theorem to answer question \ref{que1} for multi-sublinear commutators ${T_{\prod \vec b }}$.
\begin{thm}\label{thm2}
Let $m\in \mathbb{N}$, $1< p_k<\infty$, $k=1,2,\ldots,m$ with $1/p=\sum_{k=1}^m 1/{p_k}$, $\vec{\omega}=(\omega_1,\ldots,\omega_m)\in {A_{\vec P}} \cap {\left( {{A_\infty }}\right)^m}$, $v \in {A_\infty }$ and a group of non-negative measurable functions $({{\vec \varphi }_1},{\varphi _2})$ satisfy the condition:
\begin{equation}\label{con4}
{\left[ {{{\vec \varphi }_1},{\varphi _2}} \right]_2}: = \mathop {\sup }\limits_{x \in \rn,r > 0} {\varphi _2}{(x,r)^{{\rm{ - }}1}}\int_r^\infty  {{{\left( {1 + \log \frac{t}{r}} \right)}^m}\frac{{\mathop {{\rm{essinf}}}\limits_{t < \eta  < \infty } \prod\limits_{i = 1}^m {{\varphi _{1i}}(x,\eta ){\omega _i}{{(B(x,\eta ))}^{\frac{1}{{{p_i}}}}}} }}{{\prod\limits_{i = 1}^m {\omega {{(B(x,t))}^{\frac{1}{{{p_i}}}}}} }}\frac{{dt}}{t}}  < \infty.
\end{equation}
Set ${T_{\prod \vec b }}$ be a iterated commutator of $\vec b$ and m-sublinear operator $T$, where

$T \in LS\left( {\prod\limits_{i = 1}^m {{M^{{p_i},{\varphi _{1i}}}}\left( {{\omega _i}} \right)} } \right)$ and ${T_{\prod \vec b }} \in LB\left( {\prod\limits_{i = 1}^m {{M^{{p_i},{\varphi _{1i}}}}\left( {{\omega _i}} \right)}  \to {M^{p,{\varphi _2}}}\left( {{v}} \right)} \right)$.

If $\vec b \in {\left( {BMO} \right)^m}$, then ${T_{\prod \vec b }}$ is bounded from ${M^{{p_1},{\varphi _{11}}}}({\omega _1}) \times  \cdots  \times {M^{{p_m},{\varphi _{1m}}}}({\omega _m})$ to ${M^{p,{\varphi _2}}}({v})$.
Moreover, if for any $B \subseteq \rn$, $f_i \in {M^{{p_i},{\varphi _{1i}}}}({\omega _i})$, 
\begin{equation*}
{\left\| {T(f_1^0, \ldots ,f_m^0)} \right\|_{{M^{p,{\varphi _2}}(v)}}} \lesssim \prod\limits_{j = 1}^m {{{\left\| {{b_j}} \right\|}_{BMO}}} \prod\limits_{i = 1}^m {{{\left\| {{f_i}} \right\|}_{{M^{{p_i},{\varphi _{1i}}}}({\omega _i})}}},
\end{equation*}
then 
\begin{equation*}
{\left\| {{T_{\prod {\vec b} }}} \right\|_{{M^{{p_1},{\varphi _{11}}}}({\omega _1}) \times  \cdots  \times {M^{{p_m},{\varphi _{1m}}}}({\omega _m}) \to {M^{p,{\varphi _2}}}({v})}} \lesssim \prod\limits_{j = 1}^m {{{\left\| {{b_j}} \right\|}_{BMO}}},
\end{equation*}
where $f_i^0 = {f_i}{\chi _{2B}}$.
\end{thm}

\begin{center}
\textbf{2: Generalized local weighted Morrey spaces}
\end{center}

We now extend the results of Theorem \ref{thm1}, \ref{thm2} to the generalized local weighted Morrey spaces as follows whose proofs are similar to before.

\begin{thm}\label{thm7}
	Let $m\in \mathbb{N}$, $1\leq p_k<\infty$, $k=1,2,\ldots,m$ with $1/p=\sum_{k=1}^m 1/{p_k}$, $\vec{\omega}=(\omega_1,\ldots,\omega_m) \in {A_{\vec P}} \cap {\left( {{A_\infty }}\right)^m}$, $v \in {A_\infty }$ and a group of non-negative measurable functions $({{\vec \varphi }_1},{\varphi _2}) = ({\varphi _{11}}, \ldots ,{\varphi _{1m}},{\varphi _2})$ satisfy the condition:
	\begin{equation}
	{\left[ {{{\vec \varphi }_1},{\varphi _2}} \right]_1}^\prime : = \mathop {\sup }\limits_{r > 0} {\varphi _2}{({x_0},r)^{{\rm{ - }}1}}\int_r^\infty  {\frac{{\mathop {{\rm{essinf}}}\limits_{t < \eta  < \infty } \prod\limits_{i = 1}^m {{\varphi _{1i}}({x_0},\eta ){\omega _i}{{(B({x_0},\eta ))}^{\frac{1}{{{p_i}}}}}} }}{{\prod\limits_{i = 1}^m {\omega {{(B({x_0},t))}^{\frac{1}{{{p_i}}}}}} }}\frac{{dt}}{t}}  < \infty .
	\end{equation}
	Set $T$ is an m-sublinear operator on $\prod\limits_{i = 1}^m {M_{x_0}^{{p_i},{\varphi _{1i}}}}\left( {{\omega _i}} \right)$, which satisfies 
	\begin{equation}
	T \in LS\left( {\prod\limits_{i = 1}^m {{M_{x_0}^{{p_i},{\varphi _{1i}}}}\left( {{\omega _i}} \right)} } \right)\cap LB\left( {\prod\limits_{i = 1}^m {{M_{x_0}^{{p_i},{\varphi _{1i}}}}\left( {{\omega _i}} \right)}  \to {M_{x_0}^{p,{\varphi _2}}}\left( {{v}} \right)} \right).
	\end{equation}
	\begin{enumerate}[(i)]
		\item If $\mathop {\min }\limits_{1 \le k \le m} \{ {p_k}\}  > 1$, then T is bounded from ${M_{x_0}^{{p_1},{\varphi _{11}}}}({\omega _1}) \times  \cdots  \times {M_{x_0}^{{p_m},{\varphi _{1m}}}}({\omega _m})$ to ${M_{x_0}^{p,{\varphi _2}}}({v})$, i.e., $T \in B\left( {\prod\limits_{i = 1}^m {{M_{x_0}^{{p_i},{\varphi _{1i}}}}\left( {{\omega _i}} \right)}  \to {M_{x_0}^{p,{\varphi _2}}}\left( {{v}} \right)} \right);$
		\item If $\mathop {\min }\limits_{1 \le k \le m} \{ {p_k}\}  = 1$, then T is bounded from ${M_{x_0}^{{p_1},{\varphi _{11}}}}({\omega _1}) \times  \cdots  \times {M_{x_0}^{{p_m},{\varphi _{1m}}}}({\omega _m})$ to ${WM_{x_0}^{p,{\varphi _2}}}({v})$, i.e., $T \in B\left( {\prod\limits_{i = 1}^m {{M_{x_0}^{{p_i},{\varphi _{1i}}}}\left( {{\omega _i}} \right)}  \to {WM_{x_0}^{p,{\varphi _2}}}\left( {{v}} \right)} \right)$.
	\end{enumerate}
\end{thm}

\begin{thm}\label{thm8}
	Let $m\in \mathbb{N}$, $1< p_k<\infty$, $k=1,2,\ldots,m$ with $1/p=\sum_{k=1}^m 1/{p_k}$, $\vec{\omega}=(\omega_1,\ldots,\omega_m)\in {A_{\vec P}} \cap {\left( {{A_\infty }}\right)^m}$, $v \in {A_\infty }$ and a group of non-negative measurable functions $({{\vec \varphi }_1},{\varphi _2})$ satisfy the condition:
	\begin{equation}
	{\left[ {{{\vec \varphi }_1},{\varphi _2}} \right]_2}^\prime : = \mathop {\sup }\limits_{r > 0} {\varphi _2}{({x_0},r)^{{\rm{ - }}1}}\int_r^\infty  {{{\left( {1 + \log \frac{t}{r}} \right)}^m}\frac{{\mathop {{\rm{essinf}}}\limits_{t < \eta  < \infty } \prod\limits_{i = 1}^m {{\varphi _{1i}}({x_0},\eta ){\omega _i}{{(B({x_0},\eta ))}^{\frac{1}{{{p_i}}}}}} }}{{\prod\limits_{i = 1}^m {\omega {{(B({x_0},t))}^{\frac{1}{{{p_i}}}}}} }}\frac{{dt}}{t}}  < \infty .
	\end{equation}
	Set ${T_{\prod \vec b }}$ be a iterated commutator of $\vec b$ and m-sublinear operator $T$, where
	
	$T \in LS\left( {\prod\limits_{i = 1}^m {{M_{x_0}^{{p_i},{\varphi _{1i}}}}\left( {{\omega _i}} \right)} } \right)$ and ${T_{\prod \vec b }} \in LB\left( {\prod\limits_{i = 1}^m {{M_{x_0}^{{p_i},{\varphi _{1i}}}}\left( {{\omega _i}} \right)}  \to {M_{x_0}^{p,{\varphi _2}}}\left( {{v}} \right)} \right)$.
	
	If $\vec b \in {\left( {BMO} \right)^m}$, then ${T_{\prod \vec b }}$ is bounded from ${M_{x_0}^{{p_1},{\varphi _{11}}}}({\omega _1}) \times  \cdots  \times {M_{x_0}^{{p_m},{\varphi _{1m}}}}({\omega _m})$ to ${M_{x_0}^{p,{\varphi _2}}}({v})$.
	Moreover, if for any $B \subseteq \rn$, $f_i \in {M_{x_0}^{{p_1},{\varphi _{1i}}}}({\omega _i})$, 
	\begin{equation*}
{\left\| {T(f_1^0, \ldots ,f_m^0)} \right\|_{M_{{x_0}}^{p,{\varphi _2}}(v)}} \lesssim \prod\limits_{j = 1}^m {{{\left\| {{b_j}} \right\|}_{BMO}}} \prod\limits_{i = 1}^m {{{\left\| {{f_i}} \right\|}_{M_{{x_0}}^{{p_1},{\varphi _{1i}}}({\omega _i})}}},
	\end{equation*}
	then 
	\begin{equation*}
	{\left\| {{T_{\prod {\vec b} }}} \right\|_{{M_{x_0}^{{p_1},{\varphi _{11}}}}({\omega _1}) \times  \cdots  \times {M_{x_0}^{{p_m},{\varphi _{1m}}}}({\omega _m}) \to {M_{x_0}^{p,{\varphi _2}}}({v})}} \lesssim \prod\limits_{j = 1}^m {{{\left\| {{b_j}} \right\|}_{BMO}}},
	\end{equation*}
	where $f_i^0 = {f_i}{\chi _{2B}}$.
\end{thm}

\begin{center}
\textbf{3: Generalized weighted variable exponent Morrey spaces}
\end{center}

In the last topic, we extend the results of Theorem \ref{thm1}, \ref{thm2} to the generalized weighted variable exponent Morrey spaces as follows.
\begin{thm}\label{thm9}
	Let $m\in \mathbb{N}$, $p( \cdot ), p_i( \cdot ) \in LH \cap {\mathcal P}, i = 1, \cdots ,m$, with $\frac{1}{{p( \cdot )}} = \sum\limits_{i = 1}^m {\frac{1}{{{p_i}( \cdot )}}}$, $\vec{\omega}=(\omega_1,\ldots,\omega_m) \in \prod\limits_{i = 1}^m {{A_{{p_i}( \cdot )}}}$, $v \in {A_{p( \cdot )} }$ and a group of non-negative measurable functions $({{\vec \varphi }_1},{\varphi _2}) = ({\varphi _{11}}, \ldots ,{\varphi _{1m}},{\varphi _2})$ satisfy the condition:
	\begin{equation}\label{con5}
	{\left[ {{{\vec \varphi }_1},{\varphi _2}} \right]_1}^{\prime \prime }: = \mathop {\sup }\limits_{x \in \rn, r > 0} {\varphi _2}{(x,r)^{{\rm{ - }}1}}\int_r^\infty  {\frac{{\mathop {{\rm{essinf}}}\limits_{t < \eta  < \infty } \prod\limits_{i = 1}^m {{\varphi _{1i}}(x,\eta ){{\left\| {{\omega _i}} \right\|}_{L_{}^{{p_i}( \cdot )}(B(x,\eta ))}}} }}{{\prod\limits_{i = 1}^m {{{\left\| {{\omega _i}} \right\|}_{L_{}^{{p_i}( \cdot )}(B(x,t))}}} }}\frac{{dt}}{t}}  < \infty.
	\end{equation}
	Set $T$ is an m-sublinear operator on $\prod\limits_{i = 1}^m {M^{{{{p_i}( \cdot )}},{\varphi _{1i}}}}\left( {{\omega _i}} \right)$, which satisfies 
	\begin{equation}\label{cen5}
	T \in LS\left( {\prod\limits_{i = 1}^m {{M^{{{{p_i}( \cdot )}},{\varphi _{1i}}}}\left( {{\omega _i}} \right)} } \right)\cap LB\left( {\prod\limits_{i = 1}^m {{M^{{{{p_i}( \cdot )}},{\varphi _{1i}}}}\left( {{\omega _i}} \right)}  \to {M^{{{p}( \cdot )},{\varphi _2}}}\left( {{v}} \right)} \right).
	\end{equation}
	Then T is bounded from ${M^{{{{p_1}( \cdot )}},{\varphi _{11}}}}({\omega _1}) \times  \cdots  \times {M^{{{{p_m}( \cdot )}},{\varphi _{1m}}}}({\omega _m})$ to ${M^{{{p}( \cdot )},{\varphi _2}}}({v})$, i.e., 
	
	$$T \in B\left( {\prod\limits_{i = 1}^m {{M^{{{{p_i}( \cdot )}},{\varphi _{1i}}}}\left( {{\omega _i}} \right)}  \to {M^{{{p}( \cdot )},{\varphi _2}}}\left( {{v}} \right)} \right).$$
\end{thm}

\begin{thm}\label{thm10}
	Let $m\in \mathbb{N}$, $p( \cdot ), p_i( \cdot ) \in LH \cap {\mathcal P}, i = 1, \cdots ,m$, with $\frac{1}{{p( \cdot )}} = \sum\limits_{i = 1}^m {\frac{1}{{{p_i}( \cdot )}}}$, $\vec{\omega}=(\omega_1,\ldots,\omega_m) \in \prod\limits_{i = 1}^m {{A_{{p_i}( \cdot )}}}$, $v \in {A_{p( \cdot )} }$ and a group of non-negative measurable functions $({{\vec \varphi }_1},{\varphi _2})$ satisfy the condition:
	\begin{equation}\label{con6}
	{\left[ {{{\vec \varphi }_1},{\varphi _2}} \right]_2}^{\prime \prime }: = \mathop {\sup }\limits_{x \in {\rn}, r > 0} {\varphi _2}{(x,r)^{{\rm{ - }}1}}\int_r^\infty  {{{\left( {1 + \log \frac{t}{r}} \right)}^m}\frac{{\mathop {{\rm{essinf}}}\limits_{t < \eta  < \infty } \prod\limits_{i = 1}^m {{\varphi _{1i}}(x,\eta ){{\left\| {{\omega _i}} \right\|}_{L_{}^{{p_i}( \cdot )}(B(x,\eta ))}}} }}{{\prod\limits_{i = 1}^m {{{\left\| {{\omega _i}} \right\|}_{L_{}^{{p_i}( \cdot )}(B(x,t))}}} }}\frac{{dt}}{t}}  < \infty .
	\end{equation}
	Set ${T_{\prod \vec b }}$ be a iterated commutator of $\vec b$ and m-sublinear operator $T$, where
	
	$T \in LS\left( {\prod\limits_{i = 1}^m {{M^{{{p_i}( \cdot )},{\varphi _{1i}}}}\left( {{\omega _i}} \right)} } \right)$ and ${T_{\prod \vec b }} \in LB\left( {\prod\limits_{i = 1}^m {{M^{{{p_i}( \cdot )},{\varphi _{1i}}}}\left( {{\omega _i}} \right)}  \to {M^{{p}( \cdot ),{\varphi _2}}}\left( {{v}} \right)} \right)$.
	
	If $\vec b \in {\left( {BMO} \right)^m}$, then ${T_{\prod \vec b }}$ is bounded from ${M^{{{p_1}( \cdot )},{\varphi _{11}}}}({\omega _1}) \times  \cdots  \times {M^{{p_m}( \cdot ),{\varphi _{1m}}}}({\omega _m})$ to ${M^{{p}( \cdot ),{\varphi _2}}}({v})$.
	Moreover, if for any $B \subseteq \rn$, $f_i \in {M^{{{p_i}( \cdot )},{\varphi _{1i}}}}({\omega _i})$, 
	\begin{equation*}
{\left\| {T(f_1^0, \ldots ,f_m^0)} \right\|_{{M^{p( \cdot ),{\varphi _2}}}(v)}}\lesssim \prod\limits_{j = 1}^m {{{\left\| {{b_j}} \right\|}_{BMO}}} \prod\limits_{i = 1}^m {{{\left\| {{f_i}} \right\|}_{{M^{{p_i}( \cdot ),{\varphi _{1i}}}}({\omega _i})}}},
	\end{equation*}
	then 
	\begin{equation*}
	{\left\| {{T_{\prod {\vec b} }}} \right\|_{{M^{{{p_1}( \cdot )},{\varphi _{11}}}}({\omega _1}) \times  \cdots  \times {M^{{{p_m}( \cdot )},{\varphi _{1m}}}}({\omega _m}) \to {M^{{p}( \cdot ),{\varphi _2}}}({v})}} \lesssim \prod\limits_{j = 1}^m {{{\left\| {{b_j}} \right\|}_{BMO}}},
	\end{equation*}
	where $f_i^0 = {f_i}{\chi _{2B}}$.
\end{thm}

\begin{rem}
If we take $T$ to be a bilinear Calderón–Zygmund operator with Dini kernel and ${\omega _i} \equiv v \equiv 1$, then the results of boundedness still hold in Theorem $\ref{thm9}, \ref{thm10}$, which have been proved in \cite{Xu,Huang-Xu}. If we take $T$ to be a Calderón–Zygmund operator with Dini kernel, then the results of boundedness still hold in Theorem $\ref{thm9}, \ref{thm10}$, which have also been proved in \cite{Gu5}.
\end{rem}

\begin{rem}
	In Theorem $\ref{thm2}, \ref{thm8}, \ref{thm10}$, If we replace ${T_{\prod \vec b }}$ with ${T_{\sum {\vec b} }}$, the results of boundedness still holds by modifying the conditions appropriately.
\end{rem}

\begin{cor}\label{cor1}
	In Theorem $\ref{thm1}$, set $T$ is an m-sublinear operator on $\prod\limits_{i = 1}^m {M^{{p_i},{\varphi _{1i}}}}\left( {{\omega _i}} \right)$. If for any $\vec f \in \prod\limits_{i = 1}^m {M^{{p_i},{\varphi _{1i}}}}\left( {{\omega _i}} \right)$ and $x \notin \bigcap\limits_{j = 1}^m {{\rm{supp}}{f_j}}$,
	\begin{equation}\label{con1}
	\left| {T(\vec f )}(x) \right| \lesssim \int_{\nm}^{} {\frac{{\left| {\prod\limits_{i = 1}^m {{f_i}({y_i})} } \right|}}{{{{(\sum\limits_{j = 1}^m {\left| {x - {y_j}} \right|} )}^{mn}}}}d\vec y }.
	\end{equation}
	\begin{enumerate}[(i)]
		\item If $\mathop {\min }\limits_{1 \le k \le m} \{ {p_k}\}  > 1$, then $$T \in LB\left( {\prod\limits_{i = 1}^m {{M^{{p_i},{\varphi _{1i}}}}\left( {{\omega _i}} \right)}  \to {M^{p,{\varphi _2}}}\left( {{v}} \right)} \right) \Leftrightarrow T \in B\left( {\prod\limits_{i = 1}^m {{M^{{p_i},{\varphi _{1i}}}}\left( {{\omega _i}} \right)}  \to {M^{p,{\varphi _2}}}\left( {{v}} \right)} \right);$$
		\item If $\mathop {\min }\limits_{1 \le k \le m} \{ {p_k}\}  = 1$, then $$T \in LB\left( {\prod\limits_{i = 1}^m {{M^{{p_i},{\varphi _{1i}}}}\left( {{\omega _i}} \right)}  \to W{M^{p,{\varphi _2}}}\left( v \right)} \right) \Leftrightarrow T \in B\left( {\prod\limits_{i = 1}^m {{M^{{p_i},{\varphi _{1i}}}}\left( {{\omega _i}} \right)}  \to W{M^{p,{\varphi _2}}}\left( v \right)} \right).$$
	\end{enumerate}
\end{cor}

\begin{cor}\label{cor2}
	In Theorem $\ref{thm2}$, set $T$ is an m-sublinear operator on $\prod\limits_{i = 1}^m {M^{{p_i},{\varphi _{1i}}}}\left( {{\omega _i}} \right)$, which satisfies $(\ref{con1})$. If $\vec b \in {\left( {BMO} \right)^m}$, then 
	\begin{equation}
	{{T_{\prod \vec b }}}({T_{\sum {\vec b} }}) \in LB\left( {\prod\limits_{i = 1}^m {{M^{{p_i},{\varphi _{1i}}}}\left( {{\omega _i}} \right)}  \to {M^{p,{\varphi _2}}}\left( {{v}} \right)} \right) \Leftrightarrow {{T_{\prod \vec b }}}({T_{\sum {\vec b} }}) \in B\left( {\prod\limits_{i = 1}^m {{M^{{p_i},{\varphi _{1i}}}}\left( {{\omega _i}} \right)}  \to {M^{p,{\varphi _2}}}\left( {{v}} \right)} \right).
	\end{equation}
\end{cor}

\begin{rem}
The above corollaries still holds for generalized local weighted Morrey spaces and generalized weighted variable exponent Morrey spaces. In other words, boundedness on the three kinds of generalized weighted Morrey spaces is equivalent to local boundedness for $T, {{T_{\prod \vec b }}}$ and ${T_{\sum {\vec b} }}$, which satisfies $(\ref{con1})$.
\end{rem}

The organization of this paper is as follows. In section $\ref{sec2}$, we prepare some definitions and preliminary lemmas, which play a fundamental role in this paper. Section $\ref{sec3}$ is the proofs of our main results. In Section $\ref{sec4}$, we present the boundedness results for some classical multilinear operators and their commutators as applications.

Throughout this article, we let $C$ denote constants that are independent of the main parameters involved but whose value may differ from line to line. Given a ball $B$ and $\lambda>0$, $\lambda B$ denotes the ball with the same center as $B$ whose radius is $\lambda$ times that of $B$. A weight function $\omega$ is a nonnegative locally integrable function on $\mathbb R^n$ that takes values in $(0,\infty)$ almost everywhere. For a given weight function $\omega$ and a measurable set $E$, we denote the Lebesgue measure of $E$ by $|E|$ and the weighted measure of $E$ by $\omega(E)$, where $\omega(E)=\int_E \omega(x)\,dx$.
By $A\lesssim B$ we mean that $A\leq CB$ with some positive constant $C$ independent of appropriate quantities.
By $ A \approx B$, we mean that $A\lesssim B$ and $B\lesssim A$.
	\section{Some Notation and Basic Results}\label{sec2}

First let us recall some standard definitions and notations.

The classical $A_p$ weight
theory was introduced by Muckenhoupt in the study of weighted
$L^p$ boundedness of Hardy-Littlewood maximal functions, one can see Chapter 7 in \cite{Gra1}.
\begin{defn}[\cite{Gra1}]
	We denote the ball with the center $x_0$ and radius $r$ by $B=B(x_0,r)$, we say that a weight $\omega\in A_p$,\,$1<p<\infty$, if
	$$\left(\frac1{|B|}\int_B \omega(x)\,dx\right)\left(\frac1{|B|}\int_B \omega(x)^{-\frac{1}{p-1}}\,dx\right)^{p-1}\le C \quad\mbox{for every ball}\; B\subseteq \mathbb
	R^n,$$ where $C$ is a positive constant which is independent of $B$.\\
	We say a weight $\omega\in A_1$, if
	$$\frac1{|B|}\int_B \omega(x)\,dx\le C\mathop {ess\inf }\limits_{x \in B} \omega(x)\quad\mbox{for every ball}\;B\subseteq\mathbb R^n.$$
	We denote $${A_\infty } = \bigcup\limits_{1 \le p < \infty } {{A_p}}.$$
\end{defn}

Now let us recall the definitions of multiple weights. 
\begin{defn}[\cite{Lerner}]
 Let $p_1,\ldots,p_m\in[1,\infty)$ and $p\in(0,\infty)$ with $1/p=\sum_{k=1}^m 1/{p_k}$. Given $\vec{\omega}=(\omega_1,\ldots,\omega_m)$, set $v_{\vec{\omega}}=\prod_{i=1}^m \omega_i^{p/{p_i}}$. We say that $\vec{\omega}$ satisfies the $A_{\vec{P}}$ condition if it satisfies
	\begin{equation}
	\sup_B\left(\frac{1}{|B|}\int_B v_{\vec{\omega}}(x)\,dx\right)^{1/p}\prod_{i=1}^m\left(\frac{1}{|B|}\int_B \omega_i(x)^{1-{{p_i}^\prime }}\,dx\right)^{1/{{{p_i}^\prime }}}<\infty.
	\end{equation}
	when $p_i=1,$ $\left(\frac{1}{|B|}\int_B \omega_i(x)^{1-{{p_i}^\prime }}\,dx\right)^{1/{{p_i}^\prime }}$ is understood as ${(\mathop {\inf }\limits_{x \in B} {\omega _i}(x))^{ - 1}}$.
\end{defn}

\begin{lem}[\cite{Lerner}]\label{lem4}
	Let $p_1,\ldots,p_m\in[1,+\infty)$ and $1/p=\sum_{i=1}^m 1/{p_i}$. Then $\vec{\omega}\in A_{\vec{P}}$ if and only if
	\begin{equation}\label{multi2}
	\left\{
	\begin{aligned}
	&v_{\vec{\omega}}\in A_{mp},\\
	&\omega_i^{1-{{p_i}^\prime }}\in A_{m{{p_i}^\prime }},\quad i=1,\ldots,m,
	\end{aligned}\right.
	\end{equation}
	where $v_{\vec{\omega}}=\prod_{i=1}^m \omega_i^{p/{p_i}}$ and the condition $\omega_i^{1-{{p_i}^\prime }}\in A_{m{{p_i}^\prime }}$ in the case $p_i=1$ is understood as $\omega_i^{1/m}\in A_1$.
\end{lem}

\begin{lem}\label{cen4}
	Let $m\in \n,$ $p_1,\ldots,p_m\in[1,\infty)$ with $1/p=\sum_{k=1}^m 1/{p_k}$. Assume that $\omega_1,\ldots,\omega_m\in A_\infty$ and $v_{\vec{\omega}}=\prod_{i=1}^m \omega_i^{p/{p_i}}\in A_\infty$, then for any ball $B,$ we have
	\begin{equation*}\label{N1}
		\prod\limits_{i = 1}^m {{{\left( {\int_B {{\omega _i}} (x){\mkern 1mu} dx} \right)}^{p/{p_i}}}}  \approx \int_B {{v _{\vec \omega }}} (x){\mkern 1mu} dx.
	\end{equation*}
\begin{proof}[Proof:]
	Using Jensen's inequality and the definition of ${A_\infty }$ which can be found in \cite[p. 12]{Gra1} and \cite[p. 525]{Gra1}, we get
	\begin{equation*}
	\left| B \right|\exp (\frac{1}{{\left| B \right|}}\int_B {\log \omega _i^{{q_i}}} ) \le \omega _i^{{q_i}}(B) \lesssim \left| B \right|\exp (\frac{1}{{\left| B \right|}}\int_B {\log \omega _i^{{q_i}}} ).
	\end{equation*}
	and then we have
	\begin{align*}
	\prod\limits_{i = 1}^m {\omega _i^{{q_i}}{{(B)}^{\frac{q}{{{q_i}}}}}}  \approx \left| B \right|\exp (\frac{1}{{\left| B \right|}}\int_B {\log u_{\vec \omega }^q} ) \approx u_{\vec \omega }^q(B).
	\end{align*}
\end{proof}
\end{lem}

Next, we introduce some Lemmas for $BMO$ spaces. 

\begin{lem}[\cite{Gra2}]
	For all $p \in [1,\infty )$ and $b \in L_{loc}^1(\rn)$, we have
	\begin{equation*}
		\mathop {\sup }\limits_B {(\frac{1}{{\left| B \right|}}\int_B^{} {{{\left| {b(x) - {b_B}} \right|}^p}dx} )^{\frac{1}{p}}} \approx {\left\| b \right\|_{BMO}}:=\mathop {\sup }\limits_B (\frac{1}{{\left| B \right|}}\int_B^{} {\left| {b(x) - {b_B}} \right|dx} ).
	\end{equation*}
\end{lem}

\begin{lem}[\cite{Gu1}]\label{Gu1}
	Let $\omega\in A_{\infty}$ and $b \in BMO$. Then for any $p \in [1,\infty )$, $r_1, r_2 >0$, we have
	\begin{equation*}
	{\left( {\frac{1}{{\omega \left( {B({x_0},{r_1})} \right)}}\int_{B({x_0},{r_1})} {{{\left| {b(x) - {b_{B({x_0},{r_2})}}} \right|}^p}\omega (x)dx} } \right)^{\frac{1}{p}}} \lesssim {\left\| b \right\|_{BMO}}\left( {1 + \left| {\log \frac{{{r_1}}}{{{r_2}}}} \right|} \right).
	\end{equation*}
\end{lem}

\begin{lem}[\cite{Gra2}, p.166, exercises 3.1.5]\label{Gu5}
Let $b \in BMO$. For any $l \in (1,\infty )$, if $\frac{t}{r} \ge l$, then there exist two constants ${C_{n}}, {C_{n,l}}>0$, such that
\begin{equation*}
\left| {{b_{B(x,t)}} - {b_{B(x,r)}}} \right| \le {C_n}{\left\| b \right\|_{BMO}}\log (1 + \frac{t}{r}) \le {C_{n,l}}{\left\| b \right\|_{BMO}}\log \frac{t}{r}.
\end{equation*}
\end{lem}

\begin{defn}[\cite{Gu5}]
	We define the weighted variable exponent BMO spaces by
	\begin{equation*}
	BM{O_{p( \cdot ),\omega }} = \{ b:{\left\| b \right\|_{BM{O_{p( \cdot ),\omega }}}}: = \mathop {\sup }\limits_{x \in \rn,r > 0} \frac{{{{\left\| {(b( \cdot ) - {b_{B(x,r)}}){\chi _{B(x,r)}}} \right\|}_{{L^{p( \cdot )}}(\omega )}}}}{{{{\left\| {{\chi _{B(x,r)}}} \right\|}_{{L^{p( \cdot )}}(\omega )}}}}<\infty\}.
	\end{equation*}
\end{defn}

\begin{lem}[\cite{Gu5}]\label{Gu4}
	Let $p( \cdot ) \in LH \cap {\mathcal P}$. If $\omega  \in {A_{p( \cdot )}}$, then 
	${\left\| b \right\|_{BMO}} \approx {\left\| b \right\|_{BM{O_{p( \cdot ),\omega }}}}$.Moreover, we have $${\left\| {b( \cdot ) - {b_{B(x,r)}}} \right\|_{{L^{p( \cdot )}}(B(x,r),\omega dx)}} \lesssim {\left\| b \right\|_{BMO}}{\left\| \omega  \right\|_{{L^{p( \cdot )}}(B(x,r))}}.$$
\end{lem}

\begin{lem}[\cite{Gu6}]\label{Gu6}
	Let $p( \cdot ), p_i( \cdot ) \in {\mathcal P}_0, i = 1, \cdots ,m$, and $\frac{1}{{p(\cdot)}} = \sum\limits_{i = 1}^m {\frac{1}{{{p_i}(\cdot)}}}$. For any ${f_i} \in {L^{{p_i}( \cdot )}}$, we have
	\begin{equation*}
	{\left\| {{f_1} \cdots {f_m}} \right\|_{{L^{p( \cdot )}}}} \lesssim \prod\limits_{i = 1}^m {{{\left\| {{f_i}} \right\|}_{{L^{{p_i}( \cdot )}}}}}.
	\end{equation*}
\end{lem}
Now, we give the following Proposition which is crucial to our proof of Theorem \ref{thm10}.
\begin{lem}\label{cen6}
Let $q( \cdot ) \in LH \cap {\mathcal P}$, $v  \in {A_{p( \cdot )}}$. For any $l \in (1,\infty )$, if $\frac{t}{r} \ge l$, then we have 
\begin{equation*}
{\left\| {b - {b_{B(x_0,r)}}} \right\|_{{L^{q( \cdot )}}(B({x_0},t),vdx)}} \lesssim \left( {1 + \log \frac{t}{r}} \right){\left\| b \right\|_{BMO}}{\left\| v \right\|_{{L^{q( \cdot )}}(B({x_0},t))}}.
\end{equation*}
\begin{proof}[Proof:]
Combining Lemma \ref{Gu5} with Lemma \ref{Gu4}, we have
\begin{align*}
&{\left\| {b - {b_B}} \right\|_{{L^{q( \cdot )}}(B({x_0},t),vdx)}}\\
\le& {\left\| {b - {b_{B({x_0},t)}}} \right\|_{{L^{q( \cdot )}}(B({x_0},t),vdx)}} + {\left\| {{b_B} - {b_{B({x_0},t)}}} \right\|_{{L^{q( \cdot )}}(B({x_0},t),vdx)}}\\
\lesssim& {\left\| b \right\|_{BMO}}{\left\| v \right\|_{{L^{q( \cdot )}}(B({x_0},t))}} + {\left\| b \right\|_{BMO}}{\left\| v \right\|_{{L^{q( \cdot )}}(B({x_0},t))}}\log \frac{t}{r}\\
=& \left( {1 + \log \frac{t}{r}} \right){\left\| b \right\|_{BMO}}{\left\| v \right\|_{{L^{q( \cdot )}}(B({x_0},t))}}
\end{align*}
\end{proof}
\end{lem}

We will use the following statement on the boundedness of the weighted Hardy operators
\begin{equation*}
H_w^{}g(r): = \int_r^\infty  {g(t)w(t)\frac{{dt}}{t}}, H_w^*g(r): = \int_r^\infty  {{{\left( {1 + \log \frac{t}{r}} \right)}^m}g(t)w(t)\frac{{dt}}{t}}, t>0,
\end{equation*}
where $w$ is a weight. The following Lemmas are also important to prove the main results.

\begin{lem}[\cite{Gu1}]\label{Gu2}
Let $v_1, v_2$ and $w$ be weights on $\left( {0,\infty } \right)$ and $v_1$ be bounded outside a neighborhood of the origin. The inequality
\begin{equation}\label{Gu21}
\mathop {\sup }\limits_{r > 0} {v_2}(r){H_w}g(r) \lesssim \mathop {\sup }\limits_{r > 0} {v_1}(r)g(r)
\end{equation}
holds for all non-negative and non-decreasing $g$ on $\left( {0,\infty } \right)$ if and only if
\begin{equation*}
B: = \mathop {\sup }\limits_{r > 0} {v_2}(r)\int_r^\infty  {\mathop {{\rm{inf}}}\limits_{t < \eta  < \infty } \left( {{v_1}{{(\eta )}^{ - 1}}} \right)w\left( t \right)dt}  < \infty.
\end{equation*}
\end{lem}

\begin{lem}[\cite{Gu1}]\label{Gu3}
	Let $v_1, v_2$ and $w$ be weights on $\left( {0,\infty } \right)$ and $v_1$ be bounded outside a neighborhood of the origin. The inequality
	\begin{equation}\label{Gu31}
\mathop {\sup }\limits_{r > 0} {v_2}(r)H_w^*g(r) \lesssim \mathop {\sup }\limits_{r > 0} {v_1}(r)g(r)
	\end{equation}
	holds for all non-negative and non-decreasing $g$ on $\left( {0,\infty } \right)$ if and only if
	\begin{equation*}
B: = \mathop {\sup }\limits_{r > 0} {v_2}(r)\int_r^\infty  {{{\left( {1 + \log \frac{t}{r}} \right)}^m}\mathop {{\rm{inf}}}\limits_{t < \eta  < \infty } \left( {{v_1}{{(\eta )}^{ - 1}}} \right)w\left( t \right)dt}  < \infty.
	\end{equation*}
\end{lem}

\vspace{0.1cm}

\vspace{0.1cm}

\section{Proofs of Main Results}\label{sec3}

\subsection{Proofs of Theorem \ref{thm1}}
The proof of (ii) is similar to the proof of (i), so we merely conside the proof of (i).
\begin{proof}[Proof:]
Firstly, we use the piecewise integration technique to perform the following estimates. For any ${f_i} \in {M^{{p_i},{\varphi _{1i}}}}\left( {{\omega _i}} \right)$, $B=B(x_0,r)$, we have
\begin{align}\label{ieq3.1}
&\sum\limits_{j = 1}^\infty  {\prod\limits_{i = 1}^m {\frac{1}{{\left| {{2^{j + 1}}B} \right|}}} \int_{{2^{j + 1}}B} {\left| {{f_i}({y_i})} \right|d{y_i}} }\notag\\
\lesssim&\sum\limits_{j = 1}^\infty  {{{\left( {{2^{j + 1}}r} \right)}^{ - mn}}\prod\limits_{i = 1}^m {{{\left\| {{f_i}} \right\|}_{{L^{{p_i}}}({2^{j + 1}}B,{\omega _i}dx)}}{{\left\| {{\omega _i}^{ - \frac{1}{{{p_i}}}}} \right\|}_{{L^{{p_i}'}}({2^{j + 1}}B)}}} }\notag\\
\le& \sum\limits_{j = 1}^\infty  {{{\left( {{2^{j + 1}}r} \right)}^{ - mn - 1}}\int_{{2^{j + 1}}r}^{{2^{j + 2}}r} {\prod\limits_{i = 1}^m {{{\left\| {{f_i}} \right\|}_{{L^{{p_i}}}(B({x_0},t),{\omega _i}dx)}}{{\left\| {{\omega _i}^{ - \frac{1}{{{p_i}}}}} \right\|}_{{L^{{p_i}'}}(B({x_0},t))}}} dt} }\notag\\
\lesssim&\sum\limits_{j = 1}^\infty  {\int_{{2^{j + 1}}r}^{{2^{j + 2}}r} {\prod\limits_{i = 1}^m {{{\left\| {{f_i}} \right\|}_{{L^{{p_i}}}(B({x_0},t),{\omega _i}dx)}}{{\left\| {{\omega _i}^{ - \frac{1}{{{p_i}}}}} \right\|}_{{L^{{p_i}'}}(B({x_0},t))}}} \frac{{dt}}{{{t^{mn + 1}}}}} }\notag\\
\le&\int_{2r}^\infty  {\prod\limits_{i = 1}^m {{{\left\| {{f_i}} \right\|}_{{L^{{p_i}}}(B({x_0},t),{\omega _i}dx)}}{{\left\| {{\omega _i}^{ - \frac{1}{{{p_i}}}}} \right\|}_{{L^{{p_i}'}}(B({x_0},t))}}} \frac{{dt}}{{{t^{mn + 1}}}}}\notag\\
\lesssim&\int_{2r}^\infty  {\prod\limits_{i = 1}^m {{{\left\| {{f_i}} \right\|}_{{L^{{p_i}}}(B({x_0},t),{\omega _i}dx)}}{\omega _i}{{(B({x_0},t))}^{ - \frac{1}{{{p_i}}}}}} \frac{{dt}}{t}},
\end{align}
where the last step holds because of the definition of $\vec{\omega} \in {A_{\vec P}}$ and Lemma \ref{cen4}. Thus we obtain 
\begin{align}\label{ieq3.2}
&\mathop {\sup }\limits_{{x_0} \in \rn,r > 0} {\varphi _2}{({x_0},r)^{ - 1}}v{\left( B \right)^{ - \frac{1}{p}}}{\left\| {\sum\limits_{j = 1}^\infty  {\prod\limits_{i = 1}^m {\frac{1}{{\left| {{2^{j + 1}}B} \right|}}} \int_{{2^{j + 1}}B} {\left| {{f_i}({y_i})} \right|d{y_i}} } } \right\|_{{L^p}(B,vdx)}}\notag\\
\lesssim& \mathop {\sup }\limits_{{x_0} \in \rn,r > 0} {\varphi _2}{\left( {{x_0},r} \right)^{ - 1}}\int_r^\infty  {\prod\limits_{i = 1}^m {\left( {{{\left\| {{f_i}} \right\|}_{{L^{{p_i}}}(B(x_0,t),{\omega _i}dx)}}{\omega _i}{{\left( {B({x_0},t)} \right)}^{ - \frac{1}{{{p_i}}}}}} \right)\frac{{dt}}{t}}}\notag\\
\lesssim& \mathop {\sup }\limits_{{x_0} \in \rn,r > 0} \prod\limits_{i = 1}^m {\left( {{\varphi _{1i}}{{\left( {{x_0},r} \right)}^{ - 1}}{{\left\| {{f_i}} \right\|}_{{L^{{p_i}}}(B({x_0},r),{\omega _i}dx)}}{\omega _i}{{\left( {B({x_0},r)} \right)}^{ - \frac{1}{{{p_i}}}}}} \right)}\notag\\
\le& \prod\limits_{i = 1}^m {{{\left\| {{f_i}} \right\|}_{{M^{{p_i},{\varphi _{1i}}}}({\omega _i})}}},
\end{align}
where the second inequality holds since we used (\ref{con3}) and (\ref{Gu21}) in Lemma \ref{Gu2}.	

For any ${f_i} \in {M^{{p_i},{\varphi _{1i}}}}\left( {{\omega _i}} \right)$, let $f_i=f^0_i+f^{\infty}_i$, where $f^0_i=f_i\chi_{2B}$, $i=1,\ldots,m$ and $\chi_{2B}$ denotes the characteristic function of $2B$. For any $z \in \rn$, we have
\begin{equation*}
\left| {T({\vec f})(z)} \right| \le |T(f_1^0, \ldots ,f_m^0)(z)| + \sum\limits_{({\alpha _1}, \ldots ,{\alpha _m}) \ne 0} {\left| {T(f_1^{{\alpha _1}}, \cdots ,f_m^{{\alpha _m}})(z)} \right|}.
\end{equation*}

Combining $(\ref{cen2})$ with $(\ref{ieq3.2})$, we can easy to get the boundedness of $T$ as follows
\begin{align*}
&{\left\| {T(\vec f)} \right\|_{{M^{p,{\varphi _2}}}\left( {{v}} \right)}}\\
\le& {\left\| {T(f_1^0, \ldots ,f_m^0)} \right\|_{{M^{p,{\varphi _2}}}\left( {{v}} \right)}} + \sum\limits_{({\alpha _1}, \ldots ,{\alpha _m}) \ne 0} {{{\left\| {T(f_1^{{\alpha _1}}, \cdots ,f_m^{{\alpha _m}})} \right\|}_{{M^{p,{\varphi _2}}}\left( {{v}} \right)}}}\\
\le& {\left\| {T(f_1^0, \ldots ,f_m^0)} \right\|_{{M^{p,{\varphi _2}}}\left( {{v}} \right)}}\\
+& \sum\limits_{({\alpha _1}, \ldots ,{\alpha _m}) \ne 0} \mathop {\sup }\limits_{{x_0} \in \rn,r > 0}{\varphi _2}{(x_0,r)^{ - 1}}{v}{\left( B \right)^{ - \frac{1}{p}}}{\left\| {\sum\limits_{j = 1}^\infty  {\prod\limits_{i = 1}^m {\frac{1}{{\left| {{2^{j + 1}}B} \right|}}} \int_{{2^{j + 1}}B} {\left| {{f_i}({y_i})} \right|d{y_i}} } } \right\|_{{L^p}(B,vdx)}}\\
\lesssim& \prod\limits_{i = 1}^m {{{\left\| {{f_i}} \right\|}_{{M^{{p_i},{\varphi _{1i}}}}\left( {{\omega _i}} \right)}}}.
\end{align*}
\end{proof}
\subsection{Proof of Theorem \ref{thm2}}
Without loss of generality, for the sake of simplicity, we only consider the case when $m=2$.
\begin{proof}[Proof:]
For any ball $B=B(x_0,r)$, let $f_i=f^0_i+f^{\infty}_i$, where $f^0_i=f_i\chi_{2B}$, $i=1,\ldots,m$ and $\chi_{2B}$ denotes the characteristic function of $2B$. Then, we have
\begin{equation*}
\begin{split}
&{v}{\left( B \right)^{-{\frac{1}{p}}}}{\left\| {{T_{\prod {\vec b} }}({f_1},{f_2})} \right\|_{{L^p}(B,{v}dx)}}\\
\le& {v}{\left( B \right)^{ - \frac{1}{p}}}{\left\| {{T_{\prod {\vec b} }}(f_1^0,f_2^0)} \right\|_{{L^p}\left( {B,{v}dx} \right)}} + {v}{\left( B \right)^{ - \frac{1}{p}}}{\left\| {{T_{\prod {\vec b} }}(f_1^0,f_2^\infty )} \right\|_{{L^p}\left( {B,{v}dx} \right)}}\\
+&{v}{\left( B \right)^{ - \frac{1}{p}}}{\left\| {{T_{\prod {\vec b} }}(f_1^\infty ,f_2^0)} \right\|_{{L^p}\left( {B,{v}dx} \right)}} + {v}{\left( B \right)^{ - \frac{1}{p}}}{\left\| {{T_{\prod {\vec b} }}(f_1^\infty ,f_2^\infty )} \right\|_{{L^p}\left( {B,{v}dx} \right)}}\\
:=&{J_1}({x_0},r) + {J_2}({x_0},r) + {J_3}({x_0},r) + {J_4}({x_0},r).
\end{split}
\end{equation*}
We first claim that
\begin{equation}\label{eq3.5}
{J_i}({x_0},r) \le C\int_{2r}^\infty  {{{\left( {1 + \log \frac{t}{r}} \right)}^2}\prod\limits_{i = 1}^2 {\left( {{{\left\| {{f_i}} \right\|}_{{L^{{p_i}}}(B({x_0},t),{\omega _i}dx)}}{\omega _i}{{\left( {B({x_0},t)} \right)}^{ - \frac{1}{{{p_i}}}}}} \right)\frac{{dt}}{t}} }, i=1, 2, 3, 4.,
\end{equation}
where $C$ is independent of $r$, $x_0$ and $\vec f$. 

When $(\ref{eq3.5})$ are valid, the proofs of boundedness are similar to the proof ideas in Theorem \ref{thm1}, which are given as follows
\begin{align*}
&{\left\| {{{T_{\prod {\vec b} }}}(\vec f)} \right\|_{{M^{p,{\varphi _2}}}\left( {{v}} \right)}}\\
\le& \mathop {\sup }\limits_{{x_0} \in \rn,r > 0} {\varphi _2}{\left( {{x_0},r} \right)^{ - 1}}\sum {J_i\left( {x_0},r \right)}\\
\lesssim& \mathop {\sup }\limits_{{x_0} \in {\rm{\rn}},r > 0} {\varphi _2}{\left( {{x_0},r} \right)^{ - 1}}\int_r^\infty  {{{\left( {1 + {\log \frac{t}{r}}} \right)}^2}\prod\limits_{i = 1}^2 {\left( {{{\left\| {{f_i}} \right\|}_{{L^{{p_i}}}(B({x_0},t),{\omega _i}dx)}}{\omega _i}{{\left( {B({x_0},t)} \right)}^{ - \frac{1}{{{p_i}}}}}} \right)\frac{{dt}}{t}} }\\
\lesssim& \mathop {\sup }\limits_{{x_0} \in \rn,r > 0} \prod\limits_{i = 1}^2 {\left( {{\varphi _{1i}}{{\left( {{x_0},r} \right)}^{ - 1}}{{\left\| {{f_i}} \right\|}_{{L^{{p_i}}}(B({x_0},r),{\omega _i}dx)}}{\omega _i}{{\left( {B({x_0},r)} \right)}^{ - \frac{1}{{{p_i}}}}}} \right)}\\
\le& \prod\limits_{i = 1}^2 {{{\left\| {{f_i}} \right\|}_{{M^{{p_i},{\varphi _{1i}}}}({\omega _i})}}},
\end{align*}
where the third inequality holds since we used Lemma \ref{Gu3} and $(\ref{con4})$ to make (\ref{Gu31}) holds.

From the above proofs, we only need to verify the correctness of (\ref{eq3.5}).

Due to ${T_{\prod \vec b }} \in LB\left( {\prod\limits_{i = 1}^m {{M^{{p_i},{\varphi _{1i}}}}\left( {{\omega _i}} \right)}  \to {M^{p,{\varphi _2}}}\left( {{v}} \right)} \right)$, we do not need to estimate $J_1$ anymore. Note that ${J_{3}}$ is similar to ${J_{2}}$, so we merely consider to estimate ${J_{2}}$ and ${J_{4}}$.
\begin{align}\label{eq3.7}
& \left| {{T_{\prod {\vec b} }}(f_1^0,f_2^\infty )\left( z \right)} \right|\notag\\
\le& \left| {\left( {{b_1}\left( z \right) - {\mu _1}} \right)\left( {{b_2}\left( z \right) - {\mu _2}} \right)T(f_1^0,f_2^\infty )\left( z \right)} \right| + \left| {\left( {{b_1}\left( z \right) - {\mu _1}} \right)T(f_1^0,\left( {{b_2} - {\mu _2}} \right)f_2^\infty )\left( z \right)} \right|\notag\\
+& \left| {\left( {{b_2}\left( z \right) - {\mu _2}} \right)T(\left( {{b_1} - {\mu _1}} \right)f_1^0,f_2^\infty )\left( z \right)} \right| + \left| {T(\left( {{b_1} - {\mu _1}} \right)f_1^0,\left( {{b_2} - {\mu _2}} \right)f_2^\infty )\left( z \right)} \right|\notag\\ 
:=& {J_{21}}\left( z \right) + {J_{22}}\left( z \right) + {J_{23}}\left( z \right) + {J_{24}}\left( z \right),
\end{align}
where ${\mu _j} = {\left( {{b_j}} \right)_B}$.

Using H\"older's inequality and Lemma \ref{Gu1}, we have
\begin{align}\label{eq3.8}
&{\left\| {{J_{21}}} \right\|_{{L^p}\left( {B,{v}dx} \right)}}\notag\\
\le& {\left\| {\left( {{b_1} - {\mu _1}} \right)\left( {{b_2} - {\mu _2}} \right)} \right\|_{{L^p}\left( {B,{v}dx} \right)}}{\left\| {T(f_1^0,f_2^\infty )} \right\|_{{L^\infty }\left( B \right)}}\notag\\
\lesssim& \prod\limits_{i = 1}^2 {\left( {{{\left\| {{b_i} - {\mu _i}} \right\|}_{{L^{2p}}(B,{{{v}}}dx)}}} \right)} \int_{2r}^\infty  {\left( {\prod\limits_{i = 1}^m {\left( {{{\left\| {{f_i}} \right\|}_{{L^{{p_i}}}(B\left( {{x_0},t} \right),{\omega _i}dx)}}{\omega _i}{{\left( {B\left( {{x_0},t} \right)} \right)}^{ - \frac{1}{{{p_i}}}}}} \right)} } \right)\frac{{dt}}{t}}\notag\\ 
\lesssim&\left( {\prod\limits_{i = 1}^2 {{{\left\| {{b_i}} \right\|}_{BMO}}} } \right){v}{\left( B \right)^{\frac{1}{p}}}\int_{2r}^\infty  {\left( {\prod\limits_{i = 1}^m {\left( {{{\left\| {{f_i}} \right\|}_{{L^{{p_i}}}(B\left( {{x_0},t} \right),{\omega _i}dx)}}{\omega _i}{{\left( {B\left( {{x_0},t} \right)} \right)}^{ - \frac{1}{{{p_i}}}}}} \right)} } \right)\frac{{dt}}{t}}.
\end{align}
For estimating ${J_{22}}$, note that $T \in LS\left( {\prod\limits_{i = 1}^m {{M^{{p_i},{\varphi _{1i}}}}\left( {{\omega _i}} \right)} } \right)$ and we use the piecewise integration technique again to get
\begin{align}\label{eq3.9}
&\left| {T(f_1^0,\left( {{b_2} - {\mu _2}} \right)f_2^\infty )\left( z \right)} \right|\notag\\
\lesssim&\sum\limits_{j = 1}^\infty  {{{\left( {{2^{j + 1}}r} \right)}^{ - 2n}}\int_{{2^{j + 1}}B} {\int_{{{2^{j + 1}}B}} {\left| {{f_1}\left( {{y_1}} \right)\left( {{b_2}\left( {{y_2}} \right) - {\mu _2}} \right){f_2}\left( {{y_2}} \right)} \right|d{y_1}} d{y_2}} }\notag\\ 
\le&\sum\limits_{j = 1}^\infty  {{{\left( {{2^{j + 1}}r} \right)}^{ - 2n}}{{\left\| {{f_1}} \right\|}_{{L^{{p_1}}}({{2^{j + 1}}B},{\omega _1}dx)}}{{\left\| {{\omega _1}^{ - \frac{1}{{{p_1}}}}} \right\|}_{{L^{{p_1}'}}({{2^{j + 1}}B})}}{{\left\| {{f_2}} \right\|}_{{L^{{p_2}}}({2^{j + 1}}B,{\omega _2}dx)}}{{\left\| {{b_2} - {\mu _2}} \right\|}_{{L^{{p_2}'}}({2^{j + 1}}B,{\omega _2}^{1 - {p_2}'}dx)}}}\notag\\
\le& \sum\limits_{j = 1}^\infty  {{{\left( {{2^{j + 1}}r} \right)}^{ - 2n - 1}}\int_{{2^{j + 1}}r}^{{2^{j + 2}}r} {{{\left\| {{\omega _1}^{ - \frac{1}{{{p_1}}}}} \right\|}_{{L^{{p_1}'}}(B\left( {{x_0},t} \right))}}{{\left\| {{b_2} - {\mu _2}} \right\|}_{{L^{{p_2}'}}(B\left( {{x_0},t} \right),{\omega _2}^{1 - {p_2}'}dx)}}\prod\limits_{i = 1}^2 {{{\left\| {{f_i}} \right\|}_{{L^{{p_i}}}(B\left( {{x_0},t} \right),{\omega _i}dx)}}} dt} }\notag\\
\lesssim&\sum\limits_{j = 1}^\infty  {\int_{{2^{j + 1}}r}^{{2^{j + 2}}r} {{{\left\| {{\omega _1}^{ - \frac{1}{{{p_1}}}}} \right\|}_{{L^{{p_1}'}}(B\left( {{x_0},t} \right))}}{{\left\| {{b_2} - {\mu _2}} \right\|}_{{L^{{p_2}'}}(B\left( {{x_0},t} \right),{\omega _2}^{1 - {p_2}'}dx)}}\prod\limits_{i = 1}^2 {{{\left\| {{f_i}} \right\|}_{{L^{{p_i}}}(B\left( {{x_0},t} \right),{\omega _i}dx)}}} \frac{{dt}}{{{t^{2n + 1}}}}} }\notag\\
\le& \int_{2r}^\infty  {{{\left\| {{\omega _1}^{ - \frac{1}{{{p_1}}}}} \right\|}_{{L^{{p_1}'}}(B\left( {{x_0},t} \right))}}{{\left\| {{b_2} - {\mu _2}} \right\|}_{{L^{{p_2}'}}(B\left( {{x_0},t} \right),{\omega _2}^{1 - {p_2}'}dx)}}\prod\limits_{i = 1}^2 {{{\left\| {{f_i}} \right\|}_{{L^{{p_i}}}(B\left( {{x_0},t} \right),{\omega _i}dx)}}} \frac{{dt}}{{{t^{2n + 1}}}}}\notag\\
\lesssim& \left( {{{{\left\| {{b_2}} \right\|}_{BMO}}} } \right)\int_{2r}^\infty  {\left( {1 + {\log \frac{t}{r}}} \right)\left( {\prod\limits_{i = 1}^2 {\left( {{{\left\| {{f_i}} \right\|}_{{L^{{p_i}}}(B\left( {{x_0},t} \right),{\omega _i}dx)}}{\omega _i}{{\left( {B\left( {{x_0},t} \right)} \right)}^{ - \frac{1}{{{p_i}}}}}} \right)} } \right)\frac{{dt}}{t}}.
\end{align}
where the last step holds because of the definition of $\vec{\omega} \in {A_{\vec P}}$ and Lemma \ref{cen4}.
Combining with Lemma \ref{Gu1}, we can easily see that
\begin{align}\label{eq3.10}
&{\left\| {{J_{22}}} \right\|_{{L^p}\left( {B,{v}dx} \right)}}\notag\\
\le& {\left\| {\left( {{b_1} - {\mu _1}} \right)} \right\|_{{L^p}\left( {B,{v}dx} \right)}}{\left\| {T(f_1^0,\left( {{b_2} - {\mu _2}} \right)f_2^\infty )} \right\|_{{L^\infty }\left( B \right)}}\notag\\
\lesssim& {v}{\left( B \right)^{\frac{1}{p}}}\left( {\prod\limits_{i = 1}^2 {{{\left\| {{b_i}} \right\|}_{BMO}}} } \right)\int_{2r}^\infty  {\left( {1 + {\log \frac{t}{r}}} \right)\left( {\prod\limits_{i = 1}^m {\left( {{{\left\| {{f_i}} \right\|}_{{L^{{p_i}}}(B\left( {{x_0},t} \right),{\omega _i}dx)}}{\omega _i}{{\left( {B\left( {{x_0},t} \right)} \right)}^{ - \frac{1}{{{p_i}}}}}} \right)} } \right)\frac{{dt}}{t}}.
\end{align}
Similarly, we also have
\begin{align}\label{eq3.11}
&{\left\| {{J_{23}}} \right\|_{{L^p}\left( {B,{v}dx} \right)}}\notag\\
\lesssim& {v}{\left( B \right)^{\frac{1}{p}}}\left( {\prod\limits_{i = 1}^2 {{{\left\| {{b_i}} \right\|}_{BMO}}} } \right)\int_{2r}^\infty  {\left( {1 +{\log \frac{t}{r}}} \right)\left( {\prod\limits_{i = 1}^m {\left( {{{\left\| {{f_i}} \right\|}_{{L^{{p_i}}}(B\left( {{x_0},t} \right),{\omega _i}dx)}}{\omega _i}{{\left( {B\left( {{x_0},t} \right)} \right)}^{ - \frac{1}{{{p_i}}}}}} \right)} } \right)\frac{{dt}}{t}}.
\end{align}
For estimating ${J_{24}}$, we use the methods similar to getting $(\ref{eq3.9})$ to obtain
\begin{align}\label{eq3.12}
&{{J_{24}}}(z)\notag\\ 
\lesssim&\sum\limits_{j = 1}^\infty  {{{\left( {{2^{j + 1}}r} \right)}^{ - 2n}}\prod\limits_{i = 1}^2 {\int_{{2^{j + 1}}B} {\left| {\left( {{b_i}\left( {{y_i}} \right) - {\mu _i}} \right){f_i}\left( {{y_i}} \right)} \right|d{y_i}} } }\notag\\ 
\lesssim& \int_{2r}^\infty  {\prod\limits_{i = 1}^2 {{{\left\| {{f_i}} \right\|}_{{L^{{p_i}}}(B\left( {{x_0},t} \right),{\omega _i}dx)}}} {{\left\| {{b_i} - {\mu _i}} \right\|}_{{L^{{p_i}^\prime }}(B\left( {{x_0},t} \right),{\omega _i}^{1 - {p_i}^\prime }dx)}}\frac{{dt}}{{{t^{2n + 1}}}}}\notag\\
\lesssim& \left( {\prod\limits_{i = 1}^2 {{{\left\| {{b_i}} \right\|}_{BMO}}} } \right)\int_{2r}^\infty  {{{\left( {1 + {\log \frac{t}{r}}} \right)}^2}\left( {\prod\limits_{i = 1}^m {\left( {{{\left\| {{f_i}} \right\|}_{{L^{{p_i}}}(B\left( {{x_0},t} \right),{\omega _i}dx)}}{\omega _i}{{\left( {B\left( {{x_0},t} \right)} \right)}^{ - \frac{1}{{{p_i}}}}}} \right)} } \right)\frac{{dt}}{t}}.
\end{align}
The estimates of ${J_{24}}$ are given as follows
\begin{align}\label{eq3.13}
&{\left\| {{J_{24}}} \right\|_{{L^p}\left( {B,{v}dx} \right)}}\notag\\ 
\lesssim&v{\left( B \right)^{\frac{1}{p}}}\left( {\prod\limits_{i = 1}^2 {{{\left\| {{b_i}} \right\|}_{BMO}}} } \right)\int_{2r}^\infty  {{{\left( {1 + \log \frac{t}{r}} \right)}^2}\left( {\prod\limits_{i = 1}^m {\left( {{{\left\| {{f_i}} \right\|}_{{L^{{p_i}}}(B\left( {{x_0},t} \right),{\omega _i}dx)}}{\omega _i}{{\left( {B\left( {{x_0},t} \right)} \right)}^{ - \frac{1}{{{p_i}}}}}} \right)} } \right)\frac{{dt}}{t}}.
\end{align}
By using $(\ref{eq3.8})$, $(\ref{eq3.10})$, $(\ref{eq3.11})$ and $(\ref{eq3.13})$, we can obtain the estimates of $J_{2}$:
\begin{align}\label{eq3.14}
&{J_2}({x_0},r)\notag\\ 
\lesssim& \left( {\prod\limits_{i = 1}^2 {{{\left\| {{b_i}} \right\|}_{BMO}}} } \right)\int_{2r}^\infty  {{{\left( {1 + {\log \frac{t}{r}}} \right)}^2}\left( {\prod\limits_{i = 1}^m {\left( {{{\left\| {{f_i}} \right\|}_{{L^{{p_i}}}(B\left( {{x_0},t} \right),{\omega _i}dx)}}{\omega _i}{{\left( {B\left( {{x_0},t} \right)} \right)}^{ - \frac{1}{{{p_i}}}}}} \right)} } \right)\frac{{dt}}{t}}.
\end{align}
As for the estimates of ${J_{4}}$, we can first use a decomposition similar to $(\ref{eq3.7})$, and then we can make estimates similar to the above for each part separately. Last, we can obtain (\ref{eq3.5}).
\end{proof}

\subsection{Proof of Theorem \ref{thm9}}
\begin{proof}[Proof:]
	Similar to the proofs of Theorem \ref{thm1}, we still use the piecewise integration technique to perform the following estimates. For any ${f_i} \in {M^{{{{p_i}( \cdot )}},{\varphi _{1i}}}}\left( {{\omega _i}} \right)$, $B=B(x_0,r)$, by using the Lemma \ref{Gu6} multiple times, we have
	\begin{align}
	&\sum\limits_{j = 1}^\infty  {\prod\limits_{i = 1}^m {\frac{1}{{\left| {{2^{j + 1}}B} \right|}}} \int_{{2^{j + 1}}B} {\left| {{f_i}({y_i})} \right|d{y_i}} }\notag\\
	\lesssim&\sum\limits_{j = 1}^\infty  {{{\left( {{2^{j + 1}}r} \right)}^{ - mn}}\prod\limits_{i = 1}^m {{{\left\| {{f_i}} \right\|}_{{L^{{p_i}( \cdot )}}({2^{j + 1}}B,{\omega _i}dx)}}{{\left\| {{\omega _i}^{ - 1}} \right\|}_{{L^{{p_i}^\prime ( \cdot )}}({2^{j + 1}}B)}}} }\notag\\
	\le&\sum\limits_{j = 1}^\infty  {{{\left( {{2^{j + 1}}r} \right)}^{ - mn - 1}}\int_{{2^{j + 1}}r}^{{2^{j + 2}}r} {\prod\limits_{i = 1}^m {{{\left\| {{f_i}} \right\|}_{{L^{{p_i}( \cdot )}}(B({x_0},t),{\omega _i}dx)}}{{\left\| {{\omega _i}^{ - 1}} \right\|}_{{L^{{p_i}^\prime ( \cdot )}}(B({x_0},t))}}} dt} }\notag\\
	\lesssim&\sum\limits_{j = 1}^\infty  {\int_{{2^{j + 1}}r}^{{2^{j + 2}}r} {\prod\limits_{i = 1}^m {{{\left\| {{f_i}} \right\|}_{{L^{{p_i}( \cdot )}}(B({x_0},t),{\omega _i}dx)}}{{\left\| {{\omega _i}^{ - 1}} \right\|}_{{L^{{p_i}^\prime ( \cdot )}}(B({x_0},t))}}} \frac{{dt}}{{{t^{mn + 1}}}}} }\notag\\
	\le&\int_{2r}^\infty  {\prod\limits_{i = 1}^m {{{\left\| {{f_i}} \right\|}_{{L^{{p_i}( \cdot )}}(B({x_0},t),{\omega _i}dx)}}{{\left\| {{\omega _i}^{ - 1}} \right\|}_{{L^{{p_i}^\prime ( \cdot )}}(B({x_0},t))}}} \frac{{dt}}{{{t^{mn + 1}}}}}\notag\\
	\lesssim&\int_{2r}^\infty  {\prod\limits_{i = 1}^m {{{\left\| {{f_i}} \right\|}_{{L^{{p_i}( \cdot )}}(B({x_0},t),{\omega _i}dx)}}\left\| {{\omega _i}} \right\|_{{L^{{p_i}( \cdot )}}(B({x_0},t))}^{{\rm{ - }}1}} \frac{{dt}}{t}} ,
	\end{align}
	where the last step holds because of the definition of ${\omega _i} \in {A_{{p_i}( \cdot )}}$. Thus we obtain 
	\begin{align}
	&\mathop {\sup }\limits_{x_0 \in {\rn},r > 0} {\varphi _2}{({x_0},r)^{ - 1}}\left\| {{v}} \right\|_{{L^{{p_i}^\prime ( \cdot )}}(B({x_0},t))}^{{\rm{ - }}1}{\left\| {\sum\limits_{j = 1}^\infty  {\prod\limits_{i = 1}^m {\frac{1}{{\left| {{2^{j + 1}}B} \right|}}} \int_{{2^{j + 1}}B} {\left| {{f_i}({y_i})} \right|d{y_i}} } } \right\|_{{L^{{p_i}( \cdot )}}(B,{v}dx)}}\notag\\
	\lesssim&\mathop {\sup }\limits_{{x_0} \in {\rn},r > 0} {\varphi _2}{\left( {{x_0},r} \right)^{ - 1}}\int_r^\infty  {\prod\limits_{i = 1}^m {\left( {{{\left\| {{f_i}} \right\|}_{{L^{{p_i}( \cdot )}}(B({x_0},t),{\omega _i}dx)}}\left\| {{\omega _i}} \right\|_{{L^{{p_i}( \cdot )}}(B({x_0},t))}^{{\rm{ - }}1}} \right)\frac{{dt}}{t}} }\notag\\
	\lesssim&\mathop {\sup }\limits_{{x_0} \in {\rn},r > 0} \prod\limits_{i = 1}^m {\left( {{\varphi _{1i}}{{\left( {{x_0},r} \right)}^{ - 1}}{{\left\| {{f_i}} \right\|}_{{L^{{p_i}( \cdot )}}(B({x_0},r),{\omega _i}dx)}}\left\| {{\omega _i}} \right\|_{{L^{{p_i}( \cdot )}}(B({x_0},t))}^{{\rm{ - }}1}} \right)} \notag\\
	\le& \prod\limits_{i = 1}^m {{{\left\| {{f_i}} \right\|}_{{M^{{p_i}( \cdot ),{\varphi _{1i}}}}({\omega _i})}}},
	\end{align}
	where the second inequality holds since we use (\ref{con5}) and (\ref{Gu21}) in Lemma \ref{Gu2}.	
	
	The operations after that are the same as in the proofs of Theorem \ref{thm1} which we omit here.
\end{proof}

\subsection{Proof of Theorem \ref{thm10}}
Similar to proof of Theorem \ref{thm2}, for the sake of simplicity, we only consider the case when $m=2$.
\begin{proof}[Proof:]
	For any ball $B=B(x_0,r)$, let $f_i=f^0_i+f^{\infty}_i$, where $f^0_i=f_i\chi_{2B}$, $i=1,\ldots,m$ and $\chi_{2B}$ denotes the characteristic function of $2B$. Then, we have
	\begin{equation*}
	\begin{split}
	&\left\| v \right\|_{{L^{p( \cdot )}}(B({x_0},t))}^{ - 1}{\left\| {T({f_1},{f_2})} \right\|_{{L^{p( \cdot )}}(B,{v}dx)}}\\
	\le&\left\| v \right\|_{{L^{p( \cdot )}}(B({x_0},t))}^{ - 1}{\left\| {{T_{\prod {\vec b} }}(f_1^0,f_2^0)} \right\|_{{L^{p( \cdot )}}\left( {B,{v}dx} \right)}} + \left\| v \right\|_{{L^{p( \cdot )}}(B({x_0},t))}^{ - 1}{\left\| {{T_{\prod {\vec b} }}(f_1^0,f_2^\infty )} \right\|_{{L^{p( \cdot )}}\left( {B,{v}dx} \right)}}\\
	+&\left\| v \right\|_{{L^{p( \cdot )}}(B({x_0},t))}^{ - 1}{\left\| {{T_{\prod {\vec b} }}(f_1^\infty ,f_2^0)} \right\|_{{L^{p( \cdot )}}\left( {B,{v}dx} \right)}} + \left\| v \right\|_{{L^{p( \cdot )}}(B({x_0},t))}^{ - 1}{\left\| {{T_{\prod {\vec b} }}(f_1^\infty ,f_2^\infty )} \right\|_{{L^{p( \cdot )}}\left( {B,{v}dx} \right)}}\\
	:=&{L_1}({x_0},r) + {L_2}({x_0},r) + {L_3}({x_0},r) + {L_4}({x_0},r).
	\end{split}
	\end{equation*}
	We first claim that
	\begin{equation}\label{eq3.6}
	{L_i}({x_0},r) \le C\int_{2r}^\infty  {{{\left( {1 + \log \frac{t}{r}} \right)}^2}\prod\limits_{i = 1}^2 {\left( {{{\left\| {{f_i}} \right\|}_{{L^{{p_i}( \cdot )}}(B({x_0},t),{\omega _i}dx)}}\left\| {{\omega _i}} \right\|_{{L^{{p_i}( \cdot )}}(B({x_0},t))}^{ - 1}} \right)\frac{{dt}}{t}} }, i=1, 2, 3, 4.,
	\end{equation}
	where $C$ is independent of $r$, $x_0$ and $\vec f$. 
	
	When $(\ref{eq3.6})$ are valid, we have
	\begin{align*}
	&{\left\| {{{T_{\prod {\vec b} }}}(\vec f)} \right\|_{{M^{{p( \cdot )},{\varphi _2}}}\left( {{v}} \right)}}\\
	\le& \mathop {\sup }\limits_{{x_0} \in \rn,r > 0} {\varphi _2}{\left( {{x_0},r} \right)^{ - 1}}\sum {L_i\left( {x_0},r \right)}\\
	\lesssim& \mathop {\sup }\limits_{{x_0} \in {\rm{\rn}},r > 0} {\varphi _2}{\left( {{x_0},r} \right)^{ - 1}}\int_r^\infty  {{{\left( {1 + {\log \frac{t}{r}}} \right)}^2}\prod\limits_{i = 1}^2 {\left( {{{\left\| {{f_i}} \right\|}_{{L^{{p_i(\cdot)}}}(B({x_0},t),{\omega _i}dx)}}\left\| {{\omega _i}} \right\|_{{L^{{p_i}( \cdot )}}(B({x_0},t))}^{ - 1}}\right)\frac{{dt}}{t}} }\\
	\lesssim& \mathop {\sup }\limits_{{x_0} \in \rn,r > 0} \prod\limits_{i = 1}^2 {\left( {{\varphi _{1i}}{{\left( {{x_0},r} \right)}^{ - 1}}{{\left\| {{f_i}} \right\|}_{{L^{{p_i(\cdot)}}}(B({x_0},r),{\omega _i}dx)}}\left\| {{\omega _i}} \right\|_{{L^{{p_i}( \cdot )}}(B({x_0},t))}^{ - 1}}\right)}\\
	\le& \prod\limits_{i = 1}^2 {{{\left\| {{f_i}} \right\|}_{{M^{{p_i(\cdot)},{\varphi _{1i}}}}({\omega _i})}}},
	\end{align*}
	where the third inequality holds since we use Lemma \ref{Gu3} and $(\ref{con6})$ to make (\ref{Gu31}) holds.
	
	From the above proofs, we only need to verify the correctness of (\ref{eq3.6}).
	
	Due to ${T_{\prod \vec b }} \in LB\left( {\prod\limits_{i = 1}^m {{M^{{p_i(\cdot)},{\varphi _{1i}}}}\left( {{\omega _i}} \right)}  \to {M^{p(\cdot),{\varphi _2}}}\left( {{v}} \right)} \right)$, we do not need to estimate $L_1$ anymore. Note that ${L_{3}}$ is similar to ${L_{2}}$, so we merely consider to estimate ${L_{2}}$ and ${L_{4}}$.
	\begin{align}\label{ieq3.10}
	& \left| {{T_{\prod {\vec b} }}(f_1^0,f_2^\infty )\left( z \right)} \right|\notag\\
	\le& \left| {\left( {{b_1}\left( z \right) - {\mu _1}} \right)\left( {{b_2}\left( z \right) - {\mu _2}} \right)T(f_1^0,f_2^\infty )\left( z \right)} \right| + \left| {\left( {{b_1}\left( z \right) - {\mu _1}} \right)T(f_1^0,\left( {{b_2} - {\mu _2}} \right)f_2^\infty )\left( z \right)} \right|\notag\\
	+& \left| {\left( {{b_2}\left( z \right) - {\mu _2}} \right)T(\left( {{b_1} - {\mu _1}} \right)f_1^0,f_2^\infty )\left( z \right)} \right| + \left| {T(\left( {{b_1} - {\mu _1}} \right)f_1^0,\left( {{b_2} - {\mu _2}} \right)f_2^\infty )\left( z \right)} \right|\notag\\ 
	:=& {L_{21}}\left( z \right) + {L_{22}}\left( z \right) + {L_{23}}\left( z \right) + {L_{24}}\left( z \right),
	\end{align}
	where ${\mu _j} = {\left( {{b_j}} \right)_B}$.
	
	Using Lemma \ref{Gu6}, \ref{cen6} and ${\left\| {{v^{\frac{1}{2}}}} \right\|_{{L^{2p( \cdot )}}(B\left( {{x_0},t} \right))}} = \left\| v \right\|_{{L^{p( \cdot )}}(B\left( {{x_0},t} \right))}^{\frac{1}{2}}$, we have
	\begin{align}\label{ieq3.3}
	&{\left\| {{L_{21}}} \right\|_{{L^{p( \cdot )}}\left( {B,{v}dx} \right)}}\notag\\
	\le& {\left\| {\left( {{b_1} - {\mu _1}} \right)\left( {{b_2} - {\mu _2}} \right)} \right\|_{{L^{p( \cdot )}}\left( {B,{v}dx} \right)}}{\left\| {T(f_1^0,f_2^\infty )} \right\|_{{L^\infty }\left( B \right)}}\notag\\
	\lesssim&\prod\limits_{i = 1}^2 {\left( {{{\left\| {{b_i} - {\mu _i}} \right\|}_{{L^{2p(\cdot)}}(B,{v^{\frac{1}{2}}}dx)}}} \right)} \int_{2r}^\infty  {\left( {\prod\limits_{i = 1}^m {\left( {{{\left\| {{f_i}} \right\|}_{{L^{{p_i}}}(B\left( {{x_0},t} \right),{\omega _i}dx)}}\left\| {{\omega _i}} \right\|_{{L^{{p_i}( \cdot )}}(B({x_0},t))}^{ - 1}} \right)} } \right)\frac{{dt}}{t}}\notag\\ 
	\lesssim&\left( {\prod\limits_{i = 1}^2 {{{\left\| {{b_i}} \right\|}_{BMO}}} } \right){\left\| v \right\|_{{L^{p( \cdot )}}(B)}}\int_{2r}^\infty  {\left( {\prod\limits_{i = 1}^m {\left( {{{\left\| {{f_i}} \right\|}_{{L^{{p_i}}}(B\left( {{x_0},t} \right),{\omega _i}dx)}}\left\| {{\omega _i}} \right\|_{{L^{{p_i}( \cdot )}}(B({x_0},t))}^{ - 1}} \right)} } \right)\frac{{dt}}{t}}.
	\end{align}
	For estimating ${L_{22}}$, note that $T \in LS\left( {\prod\limits_{i = 1}^m {{M^{{p_i(\cdot)},{\varphi _{1i}}}}\left( {{\omega _i}} \right)} } \right)$ and we use the piecewise integration technique again to get
	\begin{align}\label{ieq3.4}
	&\left| {T(f_1^0,\left( {{b_2} - {\mu _2}} \right)f_2^\infty )\left( z \right)} \right|\notag\\
	\lesssim&\sum\limits_{j = 1}^\infty  {{{\left( {{2^{j + 1}}r} \right)}^{ - 2n}}\int_{{2^{j + 1}}B} {\int_{{{2^{j + 1}}B}} {\left| {{f_1}\left( {{y_1}} \right)\left( {{b_2}\left( {{y_2}} \right) - {\mu _2}} \right){f_2}\left( {{y_2}} \right)} \right|d{y_1}} d{y_2}} }\notag\\ 
	\lesssim&\sum\limits_{j = 1}^\infty  {{{\left( {{2^{j + 1}}r} \right)}^{ - 2n}}{{\left\| {{f_1}} \right\|}_{{L^{{p_1}( \cdot )}}({2^{j + 1}}B,{\omega _1}dx)}}{{\left\| {{\omega _1}^{ - 1}} \right\|}_{{L^{{p_1}^\prime ( \cdot )}}({2^{j + 1}}B)}}{{\left\| {{f_2}} \right\|}_{{L^{{p_2}( \cdot )}}({2^{j + 1}}B,{\omega _2}dx)}}{{\left\| {{b_2} - {\mu _2}} \right\|}_{{L^{{p_2}^\prime ( \cdot )}}({2^{j + 1}}B,{\omega _2}^{ - 1}dx)}}}\notag\\
	\le&\sum\limits_{j = 1}^\infty  {{{\left( {{2^{j + 1}}r} \right)}^{ - 2n - 1}}\int_{{2^{j + 1}}r}^{{2^{j + 2}}r} {{{\left\| {{\omega _1}^{ - 1}} \right\|}_{{L^{{p_1}^\prime ( \cdot )}}(B\left( {{x_0},t} \right))}}{{\left\| {{b_2} - {\mu _2}} \right\|}_{{L^{{p_2}^\prime ( \cdot )}}(B\left( {{x_0},t} \right),{\omega _2}^{ - 1}dx)}}\prod\limits_{i = 1}^2 {{{\left\| {{f_i}} \right\|}_{{L^{{p_i}( \cdot )}}(B\left( {{x_0},t} \right),{\omega _i}dx)}}} dt} } \notag\\
	\lesssim&\sum\limits_{j = 1}^\infty  {\int_{{2^{j + 1}}r}^{{2^{j + 2}}r} {{{\left\| {{\omega _1}^{ - 1}} \right\|}_{{L^{{p_1}^\prime ( \cdot )}}(B\left( {{x_0},t} \right))}}{{\left\| {{b_2} - {\mu _2}} \right\|}_{{L^{{p_2}^\prime ( \cdot )}}(B\left( {{x_0},t} \right),{\omega _2}^{ - 1}dx)}}\prod\limits_{i = 1}^2 {{{\left\| {{f_i}} \right\|}_{{L^{{p_i}( \cdot )}}(B\left( {{x_0},t} \right),{\omega _i}dx)}}} \frac{{dt}}{{{t^{2n + 1}}}}} }\notag\\
	\le&\int_{2r}^\infty  {{{\left\| {{\omega _1}^{ - 1}} \right\|}_{{L^{{p_1}^\prime ( \cdot )}}(B\left( {{x_0},t} \right))}}{{\left\| {{b_2} - {\mu _2}} \right\|}_{{L^{{p_2}^\prime ( \cdot )}}(B\left( {{x_0},t} \right),{\omega _2}^{ - 1}dx)}}\prod\limits_{i = 1}^2 {{{\left\| {{f_i}} \right\|}_{{L^{{p_i}( \cdot )}}(B\left( {{x_0},t} \right),{\omega _i}dx)}}} \frac{{dt}}{{{t^{2n + 1}}}}}\notag\\
	\lesssim&\left( {{{\left\| {{b_2}} \right\|}_{BMO}}} \right)\int_{2r}^\infty  {\left( {1 + \log \frac{t}{r}} \right)\left( {\prod\limits_{i = 1}^m {\left( {{{\left\| {{f_i}} \right\|}_{{L^{{p_i}( \cdot )}}(B\left( {{x_0},t} \right),{\omega _i}dx)}}\left\| {{\omega _i}} \right\|_{{L^{{p_i}^\prime ( \cdot )}}(B\left( {{x_0},t} \right))}^{ - 1}} \right)} } \right)\frac{{dt}}{t}}.
	\end{align}
	where the last step holds because of the definition of ${\omega _i} \in {A_{{p_i}( \cdot )}}$.
	Combining with Lemma \ref{Gu4}, we can easily see that
	\begin{align}\label{ieq3.5}
	&{\left\| {{L_{22}}} \right\|_{{L^{p( \cdot )}}\left( {B,{v}dx} \right)}}\notag\\
	\le& {\left\| {\left( {{b_1} - {\mu _1}} \right)} \right\|_{{L^{p( \cdot )}}\left( {B,{v}dx} \right)}}{\left\| {T(f_1^0,\left( {{b_2} - {\mu _2}} \right)f_2^\infty )} \right\|_{{L^\infty }\left( B \right)}}\notag\\
	\lesssim&{\left\| v \right\|_{{L^{p( \cdot )}}(B)}}\left( {\prod\limits_{i = 1}^2 {{{\left\| {{b_i}} \right\|}_{BMO}}} } \right)\int_{2r}^\infty  {\left( {1 + {\log \frac{t}{r}}} \right)\left( {\prod\limits_{i = 1}^m {\left( {{{\left\| {{f_i}} \right\|}_{{L^{{p_i( \cdot )}}}(B\left( {{x_0},t} \right),{\omega _i}dx)}}\left\| {{\omega _i}} \right\|_{{L^{{p_i}( \cdot )}}(B({x_0},t))}^{ - 1}} \right)} } \right)\frac{{dt}}{t}}.
	\end{align}
	Similarly, we also have
	\begin{align}\label{ieq3.6}
	&{\left\| {{L_{23}}} \right\|_{{L^{p( \cdot )}}\left( {B,{v}dx} \right)}}\notag\\
	\lesssim&{\left\| v \right\|_{{L^{p( \cdot )}}(B(x,r))}}\left( {\prod\limits_{i = 1}^2 {{{\left\| {{b_i}} \right\|}_{BMO}}} } \right)\int_{2r}^\infty  {\left( {1 + \log \frac{t}{r}} \right)\left( {\prod\limits_{i = 1}^m {\left( {{{\left\| {{f_i}} \right\|}_{{L^{{p_i( \cdot )}}}(B\left( {{x_0},t} \right),{\omega _i}dx)}}\left\| {{\omega _i}} \right\|_{{L^{{p_i}^\prime ( \cdot )}}(B\left( {{x_0},t} \right))}^{ - 1}} \right)} } \right)\frac{{dt}}{t}}.
	\end{align}
	For estimating ${L_{24}}$, combining with Lemma \ref{cen6}, we use the methods similar to getting $(\ref{ieq3.4})$ to obtain
	\begin{align}\label{ieq3.7}
	&{{L_{24}}}(z)\notag\\ 
	\lesssim&\sum\limits_{j = 1}^\infty  {{{\left( {{2^{j + 1}}r} \right)}^{ - 2n}}\prod\limits_{i = 1}^2 {\int_{{2^{j + 1}}B} {\left| {\left( {{b_i}\left( {{y_i}} \right) - {\mu _i}} \right){f_i}\left( {{y_i}} \right)} \right|d{y_i}} } }\notag\\ 
	\lesssim& \int_{2r}^\infty  {\prod\limits_{i = 1}^2 {{{\left\| {{f_i}} \right\|}_{{L^{{p_i(\cdot)}}}(B\left( {{x_0},t} \right),{\omega _i}dx)}}} {{\left\| {{b_i} - {\mu _i}} \right\|}_{{L^{{p_i}^\prime (\cdot)}}(B\left( {{x_0},t} \right),{\omega _i}^{-1}dx)}}\frac{{dt}}{{{t^{2n + 1}}}}}\notag\\
	\lesssim& \left( {\prod\limits_{i = 1}^2 {{{\left\| {{b_i}} \right\|}_{BMO}}} } \right)\int_{2r}^\infty  {{{\left( {1 + \log \frac{t}{r}} \right)}^2}\left( {\prod\limits_{i = 1}^m {\left( {{{\left\| {{f_i}} \right\|}_{{L^{{p_i}( \cdot )}}(B\left( {{x_0},t} \right),{\omega _i}dx)}}\left\| {{\omega _i}} \right\|_{{L^{{p_i}^\prime ( \cdot )}}(B\left( {{x_0},t} \right))}^{ - 1}} \right)} } \right)\frac{{dt}}{t}}.
	\end{align}
	The estimates of ${L_{24}}$ are given as follows
	\begin{align}\label{ieq3.8}
	&{\left\| {{L_{24}}} \right\|_{{L^p}\left( {B,{v}dx} \right)}}\notag\\ 
	\lesssim&{\left\| v \right\|_{{L^{p( \cdot )}}(B(x,r))}}\left( {\prod\limits_{i = 1}^2 {{{\left\| {{b_i}} \right\|}_{BMO}}} } \right)\int_{2r}^\infty  {{{\left( {1 + \log \frac{t}{r}} \right)}^2}\left( {\prod\limits_{i = 1}^m {\left( {{{\left\| {{f_i}} \right\|}_{{L^{{p_i}( \cdot )}}(B\left( {{x_0},t} \right),{\omega _i}dx)}}\left\| {{\omega _i}} \right\|_{{L^{{p_i}^\prime ( \cdot )}}(B\left( {{x_0},t} \right))}^{ - 1}} \right)} } \right)\frac{{dt}}{t}}.
	\end{align}
	By using $(\ref{ieq3.3})$, $(\ref{ieq3.5})$, $(\ref{ieq3.6})$ and $(\ref{ieq3.8})$, we can obtain the estimates of $L_{2}$:
	\begin{align}\label{ieq3.9}
	&{L_2}({x_0},r)\notag\\ 
	\lesssim&\left( {\prod\limits_{i = 1}^2 {{{\left\| {{b_i}} \right\|}_{BMO}}} } \right)\int_{2r}^\infty  {{{\left( {1 + \log \frac{t}{r}} \right)}^2}\left( {\prod\limits_{i = 1}^m {\left( {{{\left\| {{f_i}} \right\|}_{{L^{{p_i}( \cdot )}}(B\left( {{x_0},t} \right),{\omega _i}dx)}}\left\| {{\omega _i}} \right\|_{{L^{{p_i}^\prime ( \cdot )}}(B\left( {{x_0},t} \right))}^{ - 1}} \right)} } \right)\frac{{dt}}{t}}.
	\end{align}
	As for the estimates of ${L_{4}}$, we can first use a decomposition similar to $(\ref{ieq3.10})$, and then we can make estimates similar to the above for each part separately. Thus, we can obtain (\ref{eq3.6}).
\end{proof}
\section{Some Applications}\label{sec4}
To solve question \ref{que2}, in this section, we give the boundedness of some classical multilinear operators and their commutators on generalized weighted Morrey spaces as some applications of the main theorems.

\begin{center}
\textbf{I: Multilinear vector-valued Calder\'on-Zygmund operators}
\end{center}
Let $\mathscr B$ is a quasi-Banach space, $0 < p < \infty$. The $\mathscr B$-valued strongly measurable
weighted function spaces are defined by
\begin{align*}
{L^p}(\omega ,\mathscr B) =& \{ f:{\left\| f \right\|_{{L^p}(\omega ,\mathscr B)}}: = {\left\| {{{\left\| f \right\|}_\mathscr B}} \right\|_{{L^p}(\omega )}} < \infty \};\\
{M^{p,\varphi }}(\omega ,\mathscr B) =& \{ f:{\left\| f \right\|_{{L^p}(\omega ,\mathscr B)}}: = {\left\| {{{\left\| f \right\|}_\mathscr B}} \right\|_{{M^{p,\varphi }}(\omega )}} < \infty \}.
\end{align*}
Let an operator-valued function $Q:(\mm \backslash E) \to B(\cc,\mathscr B),$ $E = \{ (x,\vec y) \in \mm:x = {y_1} =  \cdots  = {y_m}\},$ we define the $m$-linear $\mathscr B$-valued Calder\'on-Zygmund operator $T$ by
\begin{equation}
T(\vec f)(x) = \int_{{\nm}} {(Q(x,\vec y))(\prod\limits_{j = 1}^m {{f_j}({y_j}))d{y_1} \cdots d{y_m}} }.
\end{equation}
for any $\vec f \in {C_c^\infty}({\rn}) \times  \cdots  \times {C_c^\infty}({\rn})$ and all $x \notin \bigcap\limits_{j = 1}^m {{\rm{supp}}{f_j}},$
and assume that $T$ can be extended to be a bounded operator from ${L^{{q_1}}} \times  \cdots {L^{{q_m}}}$ to ${L^q(\rn,\mathscr B)}$, for some $1 \le {q_1} \cdots , {q_m} \le \infty, \frac{1}{q} = \sum\limits_{k = 1}^m {\frac{1}{{{q_k}}}}>0$. The kernel $Q$ is called $m$-linear $\mathscr B$-valued Calder\'on-Zygmund kernel which satisfies, for some $\varepsilon,C>0$,
\begin{enumerate}
	\item[\emph{(i)}]${\left\| {Q(x,{y_1}, \ldots ,{y_m})} \right\|_{B\left( {\cc,\mathscr B} \right)}} \le \frac{C}{{{{(\sum\limits_{j = 1}^m | x - {y_j}|)}^{mn}}}};$
	\item[\emph{(ii)}]
	${\left\| {Q(x,{y_1}, \ldots ,{y_i}, \ldots ,{y_m}) - Q(x,{y_1}, \ldots ,{y_{i'}}, \ldots ,{y_m})} \right\|_{B\left( {\cc,\mathscr B} \right)}} \le \frac{{C|{y_i} - {y_{i'}}{|^\varepsilon }}}{{{{(\sum\limits_{j = 1}^m | x - {y_j}|)}^{mn + \varepsilon }}}}$\\
	whenever $|x-x'|\leq\frac{1}{2}\max_{1\leq j \leq m}|x-y_j|$;
	\item[\emph{(iii)}]
	${\left\| {Q(x,{y_1}, \ldots ,{y_m}) - Q(x',{y_1}, \ldots ,{y_m})} \right\|_{B\left( {\cc,\mathscr B} \right)}} \le \frac{{C|x - x'{|^\varepsilon }}}{{{{(\sum\limits_{j = 1}^m | x - {y_j}|)}^{mn + \varepsilon }}}}$\\
	whenever
	$|x-x'|\leq\frac{1}{2}\sum_{j=1}^m|x-y_j|$.
\end{enumerate}

\begin{lem}[\cite{xue6}]\label{xue6}
Let $m\in \n$ and $T$ be an $m$-linear $\mathscr B$-valued Calder\'on-Zygmund operator. Set $p_1,\ldots,p_m\in[1,\infty)$, $1/p=\sum_{k=1}^m 1/{p_k}$ and $\vec{\omega}=(\omega_1,\ldots,\omega_m)\in A_{\vec{P}}$. The following results hold:
\begin{enumerate}[(i)]
\item If $\mathop {\min }\limits_{1 \le i \le m} \{ {p_i}\}  > 1,$ then there exists a constant ${C}$, independent of $\vec f$, such that 
\begin{align*}
{\left\| {T(\vec f)} \right\|_{{L^{p}}({v_{\vec \omega }},\mathscr B)}} \le& C{\prod\limits_{i = 1}^m {\left\| {{f_i}} \right\|} _{{L^{{p_i}}}({\omega _i})}};\\
{\left\| {{T_{\prod {\vec b} }}(\vec f)} \right\|_{{L^p}({v_{\vec \omega }},\mathscr B)}} \le& C\prod\limits_{i = 1}^m {{{\left\| {{b_i}} \right\|}_{BMO}}{{\left\| {{f_i}} \right\|}_{{L^{{p_i}}}({\omega _i})}}};\\
{\left\| {{T_{\sum {\vec b} }}(\vec f)} \right\|_{{L^p}({v_{\vec \omega }},\mathscr B)}} \le& C(\mathop {\max }\limits_{1 \le i \le n} {\left\| {{b_i}} \right\|_{BMO}})\prod\limits_{i = 1}^m {{{\left\| {{f_i}} \right\|}_{{L^{{p_i}}}({\omega _i})}}}.
\end{align*}
\item If $\mathop {\min }\limits_{1 \le i \le m} \{ {p_i}\}  = 1,$ then there exists a constant ${C}$, independent of $\vec f$, such that 
\begin{equation*}
{\left\| {{T}(\vec f)} \right\|_{W{L^{p}}({v_{\vec \omega }},\mathscr B)}} \le C{\prod\limits_{i = 1}^m {\left\| {{f_i}} \right\|} _{{L^{{p_i}}}({\omega _i})}}.
\end{equation*}
\end{enumerate}
\end{lem}

\begin{rem}
The proofs of Lemma $\ref{xue6}$ are similar to the version of the scalar-valued multilinear operators, and the proofs will be mostly copies of scalar-valued ones. We omit them.
\end{rem}
For any $\vec f \in \prod\limits_{i = 1}^m {{M^{{p_1},{\varphi _{1i}}}}({\omega _i})}$, $x \in \rn$, for every ball $D \ni x$, we define the $m$-linear $\mathscr B$-valued Calder\'on-Zygmund operator $T$ by
\begin{align}\label{BCZ}
T(\vec f)(x): = T(f_1^0, \ldots ,f_m^0)(x) + \sum\limits_{({\alpha _1}, \cdots ,{\alpha _m}) \ne 0} {T(f_1^{{\alpha _1}}, \ldots ,f_m^{{\alpha _m}})(x)},
\end{align} 
where ${\alpha _j} \in \{ 0,\infty \}$, $f_i^0 = {f_i}{\chi _{2D}}$ and $f_i^\infty  = {f_i}{\chi _{{{\left( {2D} \right)}^c}}}$.

It can be confirmed by the following lemma that the above definition are well-defined and independent of the choice of $D$, whose proofs are similar to Lemma 4.1 in \cite{Gu2}, which we omit here.
\begin{lem}[\cite{Gu2}]\label{Gu7}
	Let $\vec f \in \prod\limits_{i = 1}^m {{M^{{p_1},{\varphi _{1i}}}}({\omega _i})}$ be such that the left-hand side of $(\ref{BCZ})$ converges in $\mathscr B$ for any ball $D \subseteq \rn$. If $D_1, D_2 \subseteq \rn$ and $x \in D_1 \cap D_2$, then we have
	\begin{align*}
	{T_1}(\vec f)(x) = {T_2}(\vec f)(x),
	\end{align*}
	where ${T_i}$ denote the $m$-linear $\mathscr B$-valued Calder\'on-Zygmund operator corresponding to the ball $D_i$, respectively, $i=1,2$.
\end{lem}
By Theorem \ref{thm1}, \ref{thm2} and the above arguments, our main results in this topic can be expressed as follows.
\begin{thm}\label{thm13}
	Let $m\in \mathbb{N}$, $T$ be an $m$-linear $\mathscr B$-valued Calder\'on-Zygmund operator, $1\leq p_k<\infty$, $k=1,2,\ldots,m$ with $1/p=\sum_{k=1}^m 1/{p_k}$, $\vec{\omega}=(\omega_1,\ldots,\omega_m)\in {A_{\vec P}} \cap {\left( {{A_\infty }}\right)^m}$ and a group of non-negative measurable functions $({{\vec \varphi }_1},{\varphi _2})$ satisfy the condition $(\ref{con3})$.
	\begin{enumerate}[(i)]
		\item If $\mathop {\min }\limits_{1 \le k \le m} \{ {p_k}\}  > 1$, then T is bounded from ${M^{{p_1},{\varphi _{11}}}}({\omega _1}) \times  \cdots  \times {M^{{p_m},{\varphi _{1m}}}}({\omega _m})$ to ${M^{p,{\varphi _2}}}({v_{\vec \omega }},\mathscr B)$.
		\item If $\mathop {\min }\limits_{1 \le k \le m} \{ {p_k}\}  = 1$, then T is bounded from ${M^{{p_1},{\varphi _{11}}}}({\omega _1}) \times  \cdots  \times {M^{{p_m},{\varphi _{1m}}}}({\omega _m})$ to ${WM^{p,{\varphi _2}}}({v_{\vec \omega }},\mathscr B)$.
	\end{enumerate}
\end{thm}

\begin{thm}\label{thm14}
	Let $m\in \mathbb{N}$, $T$ be an $m$-linear $\mathscr B$-valued Calder\'on-Zygmund operator, $1< p_k<\infty$, $k=1,2,\ldots,m$ with $1/p=\sum_{k=1}^m 1/{p_k}$, $\vec{\omega}=(\omega_1,\ldots,\omega_m)\in {A_{\vec P}} \cap {\left( {{A_\infty }}\right)^m}$ and a group of non-negative measurable functions $({{\vec \varphi }_1},{\varphi _2})$ satisfy the condition $(\ref{con4})$. Set ${T_{\prod \vec b }}$ be a iterated commutator of $\vec b$ and $T$.
	If $\vec b \in {\left( {BMO} \right)^m}$, then ${T_{\prod \vec b }}$ is bounded from ${M^{{p_1},{\varphi _{11}}}}({\omega _1}) \times  \cdots  \times {M^{{p_m},{\varphi _{1m}}}}({\omega _m})$ to ${M^{p,{\varphi _2}}}({v_{\vec \omega }},\mathscr B)$.
\end{thm}

\begin{thm}\label{thm15}
	Let $m\in \mathbb{N}$, $T$ be an $m$-linear $\mathscr B$-valued Calder\'on-Zygmund operator, $1< p_k<\infty$, $k=1,2,\ldots,m$ with $1/p=\sum_{k=1}^m 1/{p_k}$, $\vec{\omega}=(\omega_1,\ldots,\omega_m)\in {A_{\vec P}} \cap {\left( {{A_\infty }}\right)^m}$ and a group of non-negative measurable functions $({{\vec \varphi }_1},{\varphi _2})$ satisfy the condition:
	\begin{equation}\label{con7}
	{\left[ {{{\vec \varphi }_1},{\varphi _2}} \right]_3}: = \mathop {\sup }\limits_{x \in \rn,r > 0} {\varphi _2}{(x,r)^{{\rm{ - }}1}}\int_r^\infty  {{{\left( {1 + \log \frac{t}{r}} \right)}}\frac{{\mathop {{\rm{essinf}}}\limits_{t < \eta  < \infty } \prod\limits_{i = 1}^m {{\varphi _{1i}}(x,\eta ){\omega _i}{{(B(x,\eta ))}^{\frac{1}{{{p_i}}}}}} }}{{\prod\limits_{i = 1}^m {\omega {{(B(x,t))}^{\frac{1}{{{p_i}}}}}} }}\frac{{dt}}{t}}  < \infty.
	\end{equation}
	Set ${T_{\sum {\vec b} }}$ be a multilinear commutator of $\vec b$ and $T$. 
	If $\vec b \in {\left( {BMO} \right)^m}$, then ${T_{\sum {\vec b} }}$ is bounded from ${M^{{p_1},{\varphi _{11}}}}({\omega _1}) \times  \cdots  \times {M^{{p_m},{\varphi _{1m}}}}({\omega _m})$ to ${M^{p,{\varphi _2}}}({v_{\vec \omega }},\mathscr B)$.
\end{thm}
For the above three theorems, we only need to prove Theorem \ref{thm13} and the remaining proofs are easily obtained by the similarly methods.
\begin{proof}[Proof of Theorem \ref{thm13}:]
	We just need to prove the case of $1< p_k<\infty$, $k=1,\cdots,m$. For any ${f_i} \in {M^{{p_i},{\varphi _{1i}}}}\left( {{\omega _i}} \right)$, let $f_i=f^0_i+f^{\infty}_i$, where $f^0_i=f_i\chi_{2B}$, $i=1,\ldots,m$ and $\chi_{2B}$ denotes the characteristic function of $2B$.
	By applying Lemma \ref{xue6}, $\vec{\omega} \in {A_{\vec P}}$ and Lemma \ref{cen4}, we get
	\begin{align}\label{ieq1}
	&{\left\| {{{\left\| {T(f_1^0, \ldots ,f_m^0)} \right\|}_\mathscr B}} \right\|_{{L^p}(B,{v_{\vec \omega }}dx)}}\notag\\
	\le& {\left\| {T(f_1^0, \ldots ,f_m^0)} \right\|_{{L^p}({v_{\vec \omega }},\mathscr B)}} \lesssim \prod\limits_{i = 1}^m {{{\left\| {f_i^0} \right\|}_{{L^{{p_i}}}({\omega _i})}}} \approx {\left| B \right|^m}\int_{2r}^\infty  {\frac{{dt}}{{{t^{mn + 1}}}}} \prod\limits_{i = 1}^m {{{\left\| {{f_i}} \right\|}_{{L^{{p_i}}}(2B,{\omega _i}dx)}}}\notag\\
	\le &{\left| B \right|^m}\int_{2r}^\infty  {\prod\limits_{i = 1}^m {{{\left\| {{f_i}} \right\|}_{{L^{{p_i}}}(B\left( {{x_0},t} \right),{\omega _i}dx)}}} \frac{{dt}}{{{t^{mn + 1}}}}}\notag\\
	\lesssim &\int_{2r}^\infty  {\prod\limits_{i = 1}^m {{{\left\| {{f_i}} \right\|}_{{L^{{p_i}}}(B\left( {{x_0},t} \right),{\omega _i}dx)}}} \frac{{dt}}{{{t^{mn + 1}}}}} \prod\limits_{i = 1}^m {{\omega _i}{{\left( B \right)}^{\frac{1}{{{p_i}}}}}{{\left\| {{\omega _i}^{ - \frac{1}{{{p_i}}}}} \right\|}_{{L^{_{{p_i}'}}}(B)}}}\notag\\
	\lesssim &\int_{2r}^\infty  {\prod\limits_{i = 1}^m {\left( {{{\left\| {{f_i}} \right\|}_{{L^{{p_i}}}(B\left( {{x_0},t} \right),{\omega _i}dx)}}{{\left\| {{\omega _i}^{ - \frac{1}{{{p_i}}}}} \right\|}_{{L^{_{{p_i}'}}}(B({x_0},t))}}{{\left| B({x_0},t) \right|}^{-1}}} \right)} \frac{{dt}}{t}} \prod\limits_{i = 1}^m {{\omega _i}{{\left( B \right)}^{\frac{1}{{{p_i}}}}}}\notag\\
	\lesssim &\int_{2r}^\infty  {\prod\limits_{i = 1}^m {\left( {{{\left\| {{f_i}} \right\|}_{{L^{{p_i}}}(B\left( {{x_0},t} \right),{\omega _i}dx)}}{\omega _i}{{\left( {B\left( {{x_0},t} \right)} \right)}^{ - \frac{1}{{{p_i}}}}}} \right)} \frac{{dt}}{t}}   {v_{\vec \omega }}{\left( B \right)^{\frac{1}{p}}}.
	\end{align}
	By using (\ref{con3}) and (\ref{Gu21}) in Lemma \ref{Gu2}, we get 
	\begin{align*}
	{\left\| {T(f_1^0, \ldots ,f_m^0)} \right\|_{{M^{p,{\varphi _2}}}({v_{\vec \omega }},\mathscr B)}} \lesssim \prod\limits_{i = 1}^m {{{\left\| {{f_i}} \right\|}_{{M^{{p_i},{\varphi _{1i}}}}({\omega _i})}}}< \infty.
	\end{align*}
	
	Next, for any $1 \le l \le m$, we assume that ${\alpha _1} =  \cdots  = {\alpha _\ell } = \infty $ and ${\alpha _{l + 1}} =  \cdots  = {\alpha _m} = 0$. For any $x\in B$, we have 
	\begin{align}\label{ieq2}
	&{\left\| {T(f_1^{{\alpha _1}}, \ldots ,f_m^{{\alpha _m}})(x)} \right\|_\mathscr B}\notag\\
	\lesssim&\int_{(\mathbb R^n)^{\ell}\backslash(2B)^{\ell}}\int_{(2B)^{m-\ell}}\frac{|f_1(y_1)\cdots f_m(y_m)|}{(|x-y_1|+\cdots+|x-y_m|)^{mn}}dy_1\cdots dy_m\notag\\
	\lesssim&(\prod_{i=\ell+1}^m\int_{2B}\big|f_i(y_i)\big|\,dy_i)\times\sum_{j=1}^\infty\frac{1}{|2^{j+1}B|^m}\int_{(2^{j+1}B)^\ell\backslash(2^{j}B)^\ell}
	\big|f_1(y_1)\cdots f_{\ell}(y_\ell)\big|\,dy_1\cdots dy_\ell\notag\\
	\le&\sum_{j=1}^\infty\prod_{i=1}^m\frac{1}{|2^{j+1}B|}\int_{2^{j+1}B}\big|f_i(y_i)\big|\,dy_i,
	\end{align}
	where we used the geometric relationships: 
	if $x\in B$, $y \in {2^{j + 1}}B\backslash {2^j}B$, $j \in \n,$ then ${\left| {x - y} \right|^n} \approx \left| {{2^{j + 1}}B} \right|$
	and the sets relations: ${\left( {{{\left( {2B} \right)}^c}} \right)^l} \subseteq {\left( {{{\left( {2B} \right)}^l}} \right)^c}$. 
	
	By the above calculations, in fact, we have already proved
	\begin{equation*}
	G \in LS\left( {\prod\limits_{i = 1}^m {( {M^{{p_i},{\varphi _{1i}}}}({\omega _i}))}} \right)\cap LB\left( {\prod\limits_{i = 1}^m {({M^{{p_i},{\varphi _{1i}}}}({\omega _i}))}  \to {M^{p,{\varphi _2}}}\left( {{v_{\vec \omega }}} \right)} \right),
	\end{equation*}
	where $G(\vec f)(x): = {\left\| {T(\vec f)(x)} \right\|_\mathscr B}$. 
	
	By applying Theorem \ref{thm1}, we have 
	\begin{align*}
{\left\| {T(\vec f)} \right\|_{{M^{p,{\varphi _2}}}({v_{\vec \omega }},\mathscr B)}} = {\left\| {G(\vec f)} \right\|_{{M^{p,{\varphi _2}}}({v_{\vec \omega }})}} \lesssim \prod\limits_{i = 1}^m {{{\left\| {{f_i}} \right\|}_{{M^{p,{\varphi _{1i}}}}({\omega _i})}}}  < \infty.
	\end{align*}
	
The above estimates present that for any $\vec f \in \prod\limits_{i = 1}^m {{M^{{p_1},{\varphi _{1i}}}}({\omega _i})}$, we have ${\left\| {T(\vec f)(x)} \right\|_\mathscr B} < \infty,$ a.e. $x \in \rn$. Combining Lemma $\ref{Gu7}$, we can see that $T$ is well-defined on $\prod\limits_{i = 1}^m {{M^{{p_i},{\varphi _{1i}}}}({\omega _i})}$, which is also bounded from ${M^{{p_1},{\varphi _{11}}}}({\omega _1}) \times  \cdots  \times {M^{{p_m},{\varphi _{1m}}}}({\omega _m})$ to ${M^{p,{\varphi _2}}}({v_{\vec \omega }},\mathscr B)$. This completes the proof.
\end{proof}

\begin{center}
	\textbf{II: Multilinear Littlewood-Paley square operators}
\end{center}

\begin{defn}[\cite{xue1}]
	Let $K$ be a function defined on $\mathbb{R}^n\times \mathbb{R}^{mn}$ with $supp K\subseteq
	\mathcal{B}:=\{(x,y_1,\dots,y_m):\sum_{j=1}^m|x-y_j|^2\leq 1\}$. $K$ is called a multilinear Marcinkiewicz kernel if for some $0<\delta<mn$ and some positive constants $A$, $\gamma_0$, and $B_1$,
	\begin{enumerate}
		\item[\emph{(a)}]$
		|K(x,\vec{y})|\leq\frac{A}
		{(\sum_{j=1}^{m}|x-y_{j}|)^{mn-\delta}};
		$	\item[\emph{(b)}]$
		|K(x,\vec{y})-K(x,y_{1},\dots,y_i',\dots,y_{m})|
		\leq\frac{A|y_i-y_i'|^{\gamma_0}}
		{(\sum_{j=1}^{m}|x-y_{j}|)^{mn-\delta+\gamma_0}};$
		\item[\emph{(c)}]$
		|K(x,\vec{y})-K(x',y_{1},\dots,y_{m})|
		\leq\frac{A|x-x'|^{\gamma_0}}
		{(\sum_{j=1}^{m}|x-y_{j}|)^{mn-\delta+\gamma_0}},
		$\end{enumerate}
	where (b) holds whenever $(x,y_1,\dots,y_m)\in \mathcal{B}$ and
	$|y_i-y_i'|\leq\frac{1}{B_1}|x-y_i|$ for all $0\leq i\leq m$,
	and (c) holds whenever $(x,y_1,\dots,y_m)\in \mathcal{B}$ and
	$|x-x'|\leq\frac{1}{B_1}\max_{1\leq j \leq m}|x-y_{j}|$.
\end{defn}

\begin{defn}[\cite{xue1}]
	Let $K(x,y_1,\dots,y_m)$ be a locally integrable function defined away from the diagonal
	$x=y_1=\dots=y_m$ in $(\mathbb{R}^n)^{m+1}$.
	$K$ is called a multilinear Littlewood-Paley kernel if for some positive constants $A$, $\gamma_0$, $\delta$, and $B_1$, it holds that 
	\begin{enumerate}
		\item[\emph{(d)}]$
		|K(x,\vec{y})|\leq\frac{A}
		{(1+\sum_{j=1}^{m}|x-y_{j}|)^{mn+\delta}};
		$\item[\emph{(e)}]$
		|K(x,\vec{y})-K(x,y_{1},\dots,y_i',\dots,y_{m})|
		\leq\frac{A|y_i-y_i'|^{\gamma_0}}
		{(1+\sum_{j=1}^{m}|x-y_{j}|)^{mn+\delta+\gamma_0}}
		;$
		\item[\emph{(f)}]$
		|K(x,\vec{y})-K(x',y_{1},\dots,y_{m})|
		\leq\frac{A|x-x'|^{\gamma_0}}
		{(1+\sum_{j=1}^{m}|x-y_{j}|)^{mn+\delta+\gamma_0}},
		$\end{enumerate}
	where (e) holds whenever
	$|y_i-y_i'|\leq\frac{1}{B_1}|x-y_i|$ and for all $1\leq i\leq m$,
	and (f) holds whenever
	$|x-x'|\leq\frac{1}{B_1}\max\limits_{1\leq j \leq m}|x-y_{j}|$.
\end{defn}
\begin{defn}[\cite{xue2}]
	For any $t \in (0,\infty ),$ let ${K}(x,{y_1}, \cdots ,{y_m})$ be a locally integrable function defined away from the diagonal $x = {y_1} =  \cdots  = {y_m}$ in $(\mathbb R^n)^{m+1}$ and denote $(x,\vec y) = (x,{y_1}, \cdots ,{y_m})$, ${K_t}(x,\vec y) = \frac{1}{{{t^{mn}}}}{K}(\frac{x}{t},\frac{{{y_1}}}{t}, \ldots ,\frac{{{y_m}}}{t})$ (we will always use this notation throughout this paper). We say $K$ is a kernel of type $\theta$ if for some constants $0<\tau<1$, there exists a constant $A>0,$ such that 
	\begin{enumerate}
		\item[\emph{(g)}]$
		{(\int_0^\infty  {{{\left| {{K_t}(x,\vec y)} \right|}^2}\frac{{dt}}{t}} )^{\frac{1}{2}}} \le \frac{A}{{{{(\sum\limits_{j = 1}^m {\left| {x - {y_j}} \right|} )}^{mn}}}};$
		\item[\emph{(h)}]$
		{(\int_0^\infty  {{{\left| {{K_t}(x,\vec y) - {K_t}(x,{y_1}, \cdots ,{{y_i}'}, \cdots ,{y_m})} \right|}^2}} \frac{{dt}}{t})^{\frac{1}{2}}} \le \frac{A}{{{{(\sum\limits_{j = 1}^m {\left| {x - {y_j}} \right|} )}^{mn}}}} \cdot \theta (\frac{{{{\left| {{y_i} - {{y'}_i}} \right|} }}}{{\sum\limits_{j = 1}^m {\left| {x - {y_j}} \right|} }});$
		\item[\emph{(i)}]$
		{(\int_0^\infty  {{{\left| {{K_t}(z,\vec y) - {K_t}(x,\vec y)} \right|}^2}} \frac{{dt}}{t})^{\frac{1}{2}}} \le \frac{A}{{{{(\sum\limits_{j = 1}^m {\left| {x - {y_j}} \right|} )}^{mn}}}} \cdot \theta (\frac{{\left| {z - x} \right|}}{{\sum\limits_{j = 1}^m {\left| {x - {y_j}} \right|} }})$,
	\end{enumerate}
	where (h) holds for any $i \in \{ 1, \cdots ,m\}$, whenever $\left| {{y_i} - {y_i}^\prime } \right| \le  \frac{1}{2}\mathop {\max }\limits_{1 \le j \le m} \{ \left| {x - {y_j}} \right|\}$
	and (i) holds whenever $\left| {z - x} \right| \le \frac{1}{2} \mathop {\max }\limits_{1 \le j \le m} \{ \left| {x - {y_j}} \right|\}$.
\end{defn}
When $\theta \left( t \right) = {t^\gamma }$ for some $\gamma  > 0$, we say $K$ is a kernel of C-Z type $I$.

The multilinear square operator $T$ with kernel $K$ is defined by
\begin{equation*}\
T(\vec f)(x) = {({\int_0^\infty  {\left| {\int_{\nm} {{K_t}(x,\vec y)\prod\limits_{j = 1}^m {{f_j}({y_j})d{y_1} \cdots d{y_m}} } } \right|} ^2}\frac{{dt}}{t})^{\frac{1}{2}}},
\end{equation*}
for any $\vec f \in {C_c^\infty}({\rn}) \times  \cdots  \times {C_c^\infty}({\rn})$ and any $x \notin \bigcap\limits_{j = 1}^m {{\rm{supp}}{f_j}}$. Assume that $T$ can be extended to be a bounded operator from ${L^{{q_1}}} \times  \cdots {L^{{q_m}}}$ to ${L^q}$, for some $1 < {q_1} \cdots , {q_m} < \infty, \frac{1}{q} = \sum\limits_{k = 1}^m {\frac{1}{{{q_k}}}}$.

$T$ is called a multilinear Littlewood-Paley square operator with Dini kernel when $K$ is a kernel of type $\theta \in Dini(1)$.

$T$ is called a multilinear Marcinkiewicz integral when $K$ is a  multilinear Marcinkiewicz kernel.

$T$ is called a multilinear Littlewood-Paley $g$-function when $K$ is a multilinear Littlewood-Paley kernel.

The multilinear Littlewood-Paley $g_\lambda ^*$-function is defined by
\begin{align*}
T_{\lambda}(\vec{f})(x)=\bigg(\iint_{\mathbb{R}^{n+1}_+}\big(\frac{t}{|x-z|+t}\big)^{n\lambda}
|\int_{\mathbb{R}^{nm}}K_t(z,\vec{y})
\prod_{j=1}^mf_j(y_j)d\vec{y}|^2\frac{dzdt}{t^{n+1}}\bigg)^{\frac12},
\end{align*}
for any $\vec f \in {C_c^\infty}({\rn}) \times  \cdots  \times {C_c^\infty}({\rn})$ and any $x \notin \bigcap\limits_{j = 1}^m {{\rm{supp}}{f_j}}$, where $K$ is a multilinear Littlewood-Paley kernel. Assume that $T$ can be extended to be a bounded operator from ${L^{{q_1}}} \times  \cdots {L^{{q_m}}}$ to ${L^q}$, for some $1 < {q_1} \cdots , {q_m} < \infty, \frac{1}{q} = \sum\limits_{k = 1}^m {\frac{1}{{{q_k}}}}$.

The following Lemmas are crucial for understanding.
\begin{lem}[\cite{xue1}]
	If $K$ is either a multilinear Littlewood-Paley kernel or multilinear Marcinkiewicz kernel, then $K$ is a C-Z type $I$ kernel, furthermore, it is a Dini kernel.
\end{lem}
\begin{lem}[\cite{xue2}]\label{xue1}
	Let $m\in \n$ and $T$ be an $m$-linear square operator with Dini kernel. If $p_1,\ldots,p_m\in[1,\infty)$ with $1/p=\sum_{k=1}^m 1/{p_k}$, and $\vec{\omega}=(\omega_1,\ldots,\omega_m)\in A_{\vec{P}}$, the following results hold:
	\begin{enumerate}[(i)]
		\item If $\mathop {\min }\limits_{1 \le i \le m} \{ {p_i}\}  > 1,$ then there exists a constant ${C}$, independent of $\vec f$, such that 
		\begin{equation*}
		{\left\| {T(\vec f)} \right\|_{{L^{p}}({v_{\vec \omega }})}} \le C{\prod\limits_{i = 1}^m {\left\| {{f_i}} \right\|} _{{L^{{p_i}}}({\omega _i})}};
		\end{equation*}
		\item If $\mathop {\min }\limits_{1 \le i \le m} \{ {p_i}\}  = 1,$ then there exists a constant ${C}$, independent of $\vec f$, such that 
		\begin{equation*}
		{\left\| {T(\vec f)} \right\|_{W{L^{p}}({v_{\vec \omega }})}} \le C{\prod\limits_{i = 1}^m {\left\| {{f_i}} \right\|} _{{L^{{p_i}}}({\omega _i})}}.
		\end{equation*}
	\end{enumerate}
\end{lem}

\begin{lem}[\cite{xue5}]\label{xue2}
	Let $m\in \n$ and $T_\lambda$ be an $m$-linear Littlewood-Paley $g_\lambda ^*$-function. If $p_1,\ldots,p_m\in[1,\infty)$ with $1/p=\sum_{k=1}^m 1/{p_k}$, and $\vec{\omega}=(\omega_1,\ldots,\omega_m)\in A_{\vec{P}}$, then for any $\lambda  > 2m$, the following results hold:
	\begin{enumerate}[(i)]
		\item If $\mathop {\min }\limits_{1 \le i \le m} \{ {p_i}\}  > 1,$ then there exists a constant ${C}$, independent of $\vec f$, such that 
		\begin{equation*}\label{LEQ2}
		{\left\| {{T_\lambda }(\vec f)} \right\|_{{L^{p}}({v_{\vec \omega }})}} \le C{\prod\limits_{i = 1}^m {\left\| {{f_i}} \right\|} _{{L^{{p_i}}}({\omega _i})}};
		\end{equation*}
		\item If $\mathop {\min }\limits_{1 \le i \le m} \{ {p_i}\}  = 1,$ then there exists a constant ${C}$, independent of $\vec f$, such that 
		\begin{equation*}\label{LEQ3}
		{\left\| {{T_\lambda }(\vec f)} \right\|_{W{L^{p}}({v_{\vec \omega }})}} \le C{\prod\limits_{i = 1}^m {\left\| {{f_i}} \right\|} _{{L^{{p_i}}}({\omega _i})}}.
		\end{equation*}
	\end{enumerate}
\end{lem}
In combination with the above Lemmas, we only need to consider the multilinear square operator $T$ with Dini kernel and the multilinear Littlewood-Paley $g_\lambda ^*$-function $T_\lambda$. 
It is well known that the multilinear Littlewood-Paley square operators are important applications of the multilinear $\mathscr B$-valued Calder\'on-Zygmund operator.
By Theorem \ref{thm13}, \ref{thm14}, \ref{thm15} and the above arguments, we can formulate the main results of this topic as follows.
\begin{thm}\label{thm4}
	Let $m\in \mathbb{N}$, $T$ be an $m$-linear Littlewood-Paley square operator with Dini kernel $(T_\lambda$ be an $m$-linear Littlewood-Paley $g_\lambda ^*$-function, for any $\lambda  > 2m)$, $1\leq p_k<\infty$, $k=1,2,\ldots,m$ with $1/p=\sum_{k=1}^m 1/{p_k}$, $\vec{\omega}=(\omega_1,\ldots,\omega_m)\in {A_{\vec P}} \cap {\left( {{A_\infty }}\right)^m}$ and a group of non-negative measurable functions $({{\vec \varphi }_1},{\varphi _2})$ satisfy the condition $(\ref{con3})$.
	\begin{enumerate}[(i)]
		\item If $\mathop {\min }\limits_{1 \le k \le m} \{ {p_k}\}  > 1$, then T $\left(T_\lambda \right)$ is bounded from $\prod\limits_{i = 1}^m {M^{{p_i},{\varphi _{1i}}}({\omega _i})}$ to ${M^{p,{\varphi _2}}}({v_{\vec \omega }});$
		\item If $\mathop {\min }\limits_{1 \le k \le m} \{ {p_k}\}  = 1$, then T $\left(T_\lambda \right)$ is bounded from $\prod\limits_{i = 1}^m {M^{{p_i},{\varphi _{1i}}}({\omega _i})}$ to ${WM^{p,{\varphi _2}}}({v_{\vec \omega }}).$
	\end{enumerate}
\end{thm}

Combining Theorem \ref{thm14} and Theorem 1.3 in \cite{xue4}, we have a significant result for the iterated commutator of multilinear Littlewood–Paley $g$-function with convolution-type kernel, see \cite{xue4} for more details.
\begin{thm}\label{thm5}
	Let $m\in \mathbb{N}$, $\mathcal{G}$ be an $m$-linear Littlewood-Paley $g$-function with convolution-type kernel, $1< p_k<\infty$, $k=1,2,\ldots,m$ with $1/p=\sum_{k=1}^m 1/{p_k}$, $\vec{\omega}=(\omega_1,\ldots,\omega_m)\in {A_{\vec P}} \cap {\left( {{A_\infty }}\right)^m}$ and a group of non-negative measurable functions $({{\vec \varphi }_1},{\varphi _2})$ satisfy the condition $(\ref{con4})$. Set ${\mathcal{G}_{\prod \vec b }}$ be a iterated commutator of $\vec b$ and $\mathcal{G}$.
	If $\vec b \in {\left( {BMO} \right)^m}$, then ${\mathcal{G}_{\prod \vec b }}$ is bounded from $\prod\limits_{i = 1}^m {M^{{p_i},{\varphi _{1i}}}({\omega _i})}$ to ${M^{p,{\varphi _2}}}({v_{\vec \omega }})$.
\end{thm}

Similarly, Combining Theorem \ref{thm15} and Theorem 1.5 in \cite{He}, we also have a significant result for the multilinear commutator of multilinear Marcinkiewicz integral with convolution-type homogeneous kernel, see \cite{He} and Remark 1.1 in \cite{xue1} for more details.

\begin{thm}\label{thm6}
	Let $m\in \mathbb{N}$, $\mu$ be an $m$-linear Marcinkiewicz integral with convolution-type homogeneous kernel, $1< p_k<\infty$, $k=1,2,\ldots,m$ with $1/p=\sum_{k=1}^m 1/{p_k}$, $\vec{\omega}=(\omega_1,\ldots,\omega_m)\in {A_{\vec P}} \cap {\left( {{A_\infty }}\right)^m}$ and a group of non-negative measurable functions $({{\vec \varphi }_1},{\varphi _2})$ satisfy the condition $(\ref{con7})$.
	Set ${\mu_{\sum {\vec b} }}$ be a multilinear commutator of $\vec b$ and $\mu$. 
	If $\vec b \in {\left( {BMO} \right)^m}$, then ${\mu_{\sum {\vec b} }}$ is bounded from $\prod\limits_{i = 1}^m {M^{{p_i},{\varphi _{1i}}}({\omega _i})}$ to ${M^{p,{\varphi _2}}}({v_{\vec \omega }})$.
\end{thm}

\begin{center}
\textbf{III: multilinear pseudo-differential operators and multilinear paraproducts}
\end{center}
\begin{lem}[\cite{Gra2}, p.555, exercises 7.4.4]\label{PD}
	Let $\sigma$ be a smooth function on $(\rn)^{m+1}$ satisfying 
	\begin{align*}
	\left| {\partial _x^\alpha \partial _{{\xi _1}}^{{\beta _1}} \cdots \partial _{{\xi _m}}^{{\beta _m}}\sigma (x,\vec \xi )} \right| \le {C_{\alpha ,\beta }}{(1 + \left| {{\xi _1}} \right| +  \cdots  + \left| {{\xi _m}} \right|)^{\left| \alpha  \right| - (\left| {{\beta _1}} \right| +  \cdots  + \left| {{\beta _m}} \right|)}},
	\end{align*}
	for all $\alpha ,{\beta _1}, \ldots ,{\beta _m}$ $n$-tuples of nonnegative integers. Let $\mathop \sigma \limits^ \vee  (x,\vec z)$ be the inverse Fourier transform of $\sigma (x,\vec \xi )$ in the ${\vec \xi }$ variable, $K({y_0}, \ldots ,{y_m}) = \mathop \sigma \limits^ \vee  ({y_0},{y_0} - {y_1}, \ldots ,{y_0} - {y_m})$. The m-linear pseudodifferential operator is defined by
	\begin{align*}
	{T_\sigma }({f_1}, \cdots ,{f_m})(x) = \int_{\nm} {\sigma (x,\vec \xi )\widehat {{f_1}}({\xi _1}) \cdots \widehat {{f_m}}({\xi _m}){e^{2\pi ix \cdot ({\xi _1} +  \cdots  + {\xi _m})}}d{\xi _1} \cdots d{\xi _m}}.
	\end{align*}
	Suppose that all of the transposes ${T^{*j}}$ also have symbols that satisfy the same estimates as $\sigma$. Then we have 
	\begin{align*}
	\left| {\partial _{{y_0}}^{{\alpha _0}} \cdots \partial _{{y_0}}^{{\alpha _m}}K({y_0}, \ldots ,{y_m})} \right| \le \frac{{{C_{{\alpha _0}, \ldots ,{\alpha _m}}}}}{{{{(\sum\limits_{i = 1}^m {\left| {{y_0} - {y_i}} \right|} )}^{mn + \left| {{\alpha _0}} \right| +  \cdots  + \left| {{\alpha _m}} \right|}}}},
	\end{align*}
	in particular, K is a $m$-linear Calder\'on-Zygmund kernel. Furthermore, ${T_\sigma }$ can extend as bounded operator from ${L^{{q_1}}} \times  \cdots  \times {L^{{q_m}}}$ to $L^q$, for ${q_j} \in (1,\infty )$ and $\frac{1}{q} = \sum\limits_{i = 1}^m {\frac{1}{{{q_i}}}}$. From this, we can conclude ${T_\sigma }$ is a $m$-linear Calder\'on-Zygmund operator.
\end{lem}

\begin{lem}[\cite{Gra2}, p.556, exercises 7.4.5 and 7.4.6]\label{Pa}
	Let $\Psi  \in \mathscr{S}({\rn})$ whose Fourier transform is supported in the annulus $\frac{6}{7} \le \left| \xi  \right| \le 2$ and is equal to $1$ on $1 \le \left| \xi  \right| \le \frac{{12}}{7}$. Set $\Delta _j^\Psi (f){\rm{ = }}f * {\Psi _{{2^{ - j}}}}$ and ${S_j}(f) = \sum\limits_{k \le j} {\Delta _j^\Psi (f)}$. For $f,g \in \mathscr{S}(\rn)$, the bilinear and trilinear paraproducts are defined by
	\begin{align*}
	{\Pi _2}({f_1},{f_2}) =& \sum\limits_{j \in \mathbb{Z}} {\Delta _j^\Psi ({f_1})} {S_j}({f_2})\\
	{\Pi _3}({f_1},{f_2},{f_3}) =& \sum\limits_{j \in \mathbb{Z}} {\Delta _j^\Psi ({f_1})} {S_j}({f_2}){S_j}({f_3})
	\end{align*}
	Then the kernel of ${\Pi _i}$ is a i-linear Calder\'on-Zygmund kernel, $i=2,3$. Moreover, we can conclude ${\Pi _i}$ is a $i$-linear Calder\'on-Zygmund operator, $i=2,3$.
\end{lem}
From the above Lemmas, we can directly obtain the following results.
\begin{cor}\label{cor3}
	With the notation of Lemma $\ref{PD}$ and $\ref{Pa}$, then the results of Theorem $\ref{thm13}$, $\ref{thm14}$ and $\ref{thm15}$ still hold for ${T_\sigma }$, ${\Pi _2}$ and ${\Pi _3}$.
\end{cor}

\begin{rem}
	In Theorem $\ref{thm13}, \ref{thm14}, \ref{thm15}, \ref{thm4}, \ref{thm5}, \ref{thm6}$ and Corollary $\ref{cor3}$, if we replace ${M^{{p_i},{\varphi _{11i}}}}({\omega _i})$, ${M^{p,{\varphi _2}}}({v_{\vec \omega }})$ by ${M_{x_0}^{{p_i},{\varphi _{11i}}}}({\omega _i})$ and ${M_{x_0}^{p,{\varphi _2}}}({v_{\vec \omega }})$, then according to Theorem $\ref{thm7}, \ref{thm8}$, the similar results still hold for generalized local weighted Morrey spaces.
\end{rem}
\vspace{0.4cm}

\medskip 

\noindent{\bf Data Availability} Our manuscript has no associated data.

\medskip 
\noindent{\bf\Large Declarations}
\medskip 

\noindent{\bf Conflict of interest} The authors state that there is no conflict of interest.

\bigskip

\noindent Xi Cen

\smallskip

\noindent {\it Address:} School of Mathematics and Physics, Southwest University of Science and Technology, Mianyang, 621010, P. R. China

\smallskip

\noindent {\it E-mail:} xicenmath@gmail.com

\bigskip

\noindent Xiang Li

\smallskip

\noindent {\it Address:} School of Science, Shandong Jianzhu University, Jinan, 250000, P. R. China

\smallskip

\noindent {\it E-mail:} lixiang162@mails.ucas.ac.cn

\bigskip

\noindent Dunyan Yan

\smallskip

\noindent {\it Address:} School of Mathematical Sciences, University of Chinese Academy of Sciences, Beijing, 100049, P. R. China

\smallskip

\noindent {\it E-mail:} ydunyan@ucas.ac.cn
\end{document}